\documentclass{amsart}
\usepackage{graphicx}
\usepackage{graphics, epsfig}

\newtheorem{thm}{Theorem}[section]

\newtheorem{lem}[thm]{Lemma}
\newtheorem{prop}[thm]{Proposition}
\theoremstyle{definition}
\newtheorem{defn}[thm]{Definition}
\theoremstyle{remark}
\newtheorem{rem}[thm]{Remark}
\newtheorem{ex}[thm]{Example}
\numberwithin{equation}{section}

\begin{document}

\title[State-crossings in random walks]
{Crossings states and sets of states in random walks}%
\author{Vyacheslav M. Abramov}%
\address{24 Sagan Drive, North Cranbourne, Victoria - 3977, Australia}%
\email{vabramov126@gmail.com}%


\begin{abstract}
We consider a random walk in $\mathbb{Z}^2$ and
establish a number of results for the
distributions and expectations of the number of
usual (undirected) and specifically defined in the paper directed state-crossings and different sets of states
crossings. As well, we extend the results to $d$-dimensional random walks, $d\geq2$, in bounded areas.
\end{abstract}
\maketitle

\section{Introduction}
\subsection{Formulation of the problem and the literature} In this paper we study
simple random walks, a well-known object of the study in
probability theory \cite{Spitzer}. The known results on simple random walks can be found in many sources (e. g. \cite{DoyleSnell, Hughes, LyonsPeres, Spitzer}).

The present paper studies the probability distributions and expectations of the number of usual
(undirected) and specifically defined (directed) state-crossings for random walks in $\mathbb{Z}^2$.
The definitions of these notions are given later in the paper. We also consider multidimensional random walks in bounded areas.

The time parameter $t$ is a discrete (integer) parameter. At time
$t=0$, the random walk starts from $\mathbf{0}$, where $\mathbf{0}=(0, 0)$, and, after some
excursion, it visits $\mathbf{0}$ again for the first time after
$t=0$. This random stopping time of the excursion is denoted by
$\tau$. For the random walk in $\mathbb{Z}^2$ it exists with probability 1 (see e. g. \cite{Spitzer}).

Vectors and their components are denoted as follows: $\mathbf{n}=
\big(n^{(1)}, n^{(2)}\big)$, $\mathbf{n}_i=\big(n^{(1)}_i, n^{(2)}_i\big)$,
$\mathbf{S}_t=\big(S^{(1)}_t, S^{(2)}_t\big)$ and so on, by using bold font Latin (lower or upper case) letters for vectors and the italic font Latin
letters with order number indices for components. The words
\textit{vector} and \textit{state} are used interchangeably.

The random walks in $\mathbb{Z}^2$ are defined as follows.
The vector $\mathbf{S}_t=\big(S_t^{(1)}, S_t^{(2)}\big)$ denotes the state of the random walk at time $t$
and is defined recurrently as follows:
\begin{eqnarray}
\mathbf{S}_0&=&\mathbf{0},\label{eq-1.1}\\
\mathbf{S}_{t}&=&\mathbf{S}_{t-1}+\mathbf{e}_t(\mathbb{Z}^2), \quad
t\geq1,\label{eq-1.2}
\end{eqnarray}
where the random vector $\mathbf{e}_t(\mathbb{Z}^2)$ is in turn
defined as follows. Let $\mathbf{1}_{i}$ denote the vector, the
$i$th component of which is 1, and the rest component are 0. Then,
the vector $\mathbf{e}_t(\mathbb{Z}^2)$ is one of the $4$ vectors
$\{\pm\mathbf{1}_{i}, i=1,2\}$ that is randomly chosen with
probability $1/4$ independently of the other vectors and the
history of the random walk.

The reason for the unusual notation $\mathbf{e}_t(\mathbb{Z}^2)$ is that,
the recurrence relations similar to \eqref{eq-1.1} describe a variety of different random walks,
in which there is only the difference in the definition of the vector $\mathbf{e}_t(\bullet)$.
Specifically, in the place of $\bullet$ one can use the subsets of  $\mathbb{Z}^2$ for which
the random walk is defined. For instance,
$\mathbf{e}_t(\mathbb{Z}^2_+)$, $\mathbf{e}_t([-N,N]^d)$ or
$\mathbf{e}_t([0,N]^d)$ is the notation for the increments of the
random walks in $\mathbb{Z}^2_+$, $[-N,N]^d$ or $[0,N]^d$,
respectively.

Level-crossings are widely known in probability theory and its applications, and there is an extensive bibliography on this subject. Recent paper on this subject related to the theory of random walks and diffusions can be found in
\cite{Loch, MV1, MV2}.
The aim of the present paper is to provide direct studies for the distributions of state-crossings and crossings special sets of states. As well, we obtain the results for the expectations.

The initial point of our study are level-crossings for the symmetric one-dimensional random walk. Level-crossings for that random walk are mentioned in a number of sources. The known result is the following property. Suppose that the random walk starts at 0 and, after some excursion, at the first time returns to zero again. \textit{Then for
any level $i\neq0$, the expected number of crossings level $i$
is the same for all $i$ and equal to $1$.} This result can be found in a number of textbooks such as \cite{F, Durrett, Szekeley, Wolff}. Following the title of the book by Szek\'ely \cite{Szekeley}, the aforementioned result is specified as a paradox in probability theory.

Note that the continuous time analogue of the symmetric one-dimensional random walk reflected at zero is the $M/M/1$ queueing system with equal interarrival and service times means. For the series of finite capacity $M/M/1/n$ queueing systems, the aforementioned property of random walk can be reformulated as follows. \textit{If
the expectations of interarrival and service times are equal, then
the expected number of losses during a busy period in the $M/M/1/n$
queueing system is equal to 1 for any $n$.} Surprisingly, the last claim holds true for the series of $M/G/1/n$ queueing systems with generally distributed service times \cite{Abramov4, Abramov1, Abramov2, Righter, Wolff2}. For characterization of the properties of losses in general queues see \cite{Abramov3}.

\subsection{Methodology and contribution} We adapt the level-crossings method originally used in \cite{Abramov5} and developed in \cite{Abramov6, Abramov7} for multidimensional random walks. One-dimensional particular case considered in the paper demonstrates an example for the further development of the method in the multidimensional cases. By using that level-crossings method we establish connection between the distributions of number of state-crossings and certain sets of states crossings with the distributions of the random fields defined in the paper. For the symmetric one-dimensional random walk the distribution of the number of level-crossings is expressed via the distribution of the number of offspring in a certain birth-and-death process. Although the results in the one-dimensional case cannot be considered as new, the connection between the number of level-crossings and birth-and-death processes is a novel knowledge that has not been mentioned before. The results for two- and multi-dimensional random walks all are new. The distributions of the number of state-crossings and crossings sets of states are expressed via the distributions of the special random fields introduced in the paper. It is proved that the expected number of state-crossings is equal to one for any state $\mathbf{n}$. As well, the expected number of directed state-crossings from the above and below are obtained. These results are derived based on the approach in \cite{Abramov_AoP}.

\subsection{Historical background and possible applications} Level-crossings approach goes back to the classic \textit{Doob's martingale convergence theorem} under the titles \textit{up-crossings} and \textit{down-crossings} (see \cite{Doob}). Level-crossings applications in probability theory, and in particular in queueing theory and dam systems were initiated in 1970th in a number of papers by Brill and Posner (e. g. \cite{Brill1, BrillPos}) and summarised in \cite{Brill} and by Cohen \cite{Cohen1, Cohen2}. Application to inventory systems can be found in \cite{AzouryBrill} and to credit risk models in \cite{JPS, Sezer}. Applications in statistics of stochastic processes for parameter estimates, testing hypothesis and simulation can be found in \cite{BurqJones, JonesShen, RollsJones}. Level-crossings of stationary Gaussian processes reviewed in \cite{Kratz}. The further references can be found there.

A special circle of problems associated with level-crossings of state-dependent Markovian queueing systems with or without losses has been considered in \cite{Abramov5}. Application of level-crossings to epidemic models has been given in \cite{Abramov8}. Being further developed in the present paper for multi-dimensional random walks, the method may have a number of novel applications that include complex queueing systems with losses, in which the losses from a queue can depend on parameters, the behavior of which is described by two- or multi-dimensional random walk. Another possible area of applications is multi-type epidemic models that generalize the models studied in \cite{Abramov8}.

\subsection{Outline of the paper}
The rest of the paper is organized as follows. In Section \ref{S2},
we provide necessary notation, definitions, remarks and examples. In Section \ref{S3}, we formulate the main results of the paper. In Section \ref{S4}, we formulate and prove the claim for level-crossings in the one-dimensional random walk $S_t$. The proof for the one-dimensional random walk is important, since it makes the further proofs of the main results more understandable and clearer. In Sections \ref{S5}, \ref{S6}, \ref{S7} and \ref{S8} we prove the main results of the paper. The proof of Theorem \ref{t4} given in Section \ref{S8} is independent of the proofs of Theorems \ref{t2} and \ref{t1}. In Section \ref{S9}, the results of the paper are developed for $d$-dimensional random walks in bounded areas. In Section \ref{S10}, we discuss crossings sets and states for random walks in $\mathbb{Z}^d$ for $d\geq3$ and formulate a conjecture on the behaviour of the expected state-crossings.
In Section \ref{S_num} we provide some numerical studies for the crossings states and sets of states for the two-dimensional random walk.  In Section \ref{S11}, we resume the part of the results of the paper related to the expectations of crossings sets and states of sets. To make the paper self-contained we added the Appendix, the results of which are used to prove some key results of the paper.

\section{Notation, definitions, remarks and examples}\label{S2}

The following notation is used in the paper. For any vector
$\mathbf{n}\in\mathbb{Z}^2$, its $l_1$-norm is
\begin{equation*}\label{eq-2.1}
\| \mathbf{n} \|=|n^{(1)}|+|n^{(2)}|
\end{equation*}

For $\mathbf{n}_1$, $\mathbf{n}_2,\ldots$, $\mathbf{n}_l\in\mathbb{Z}^2$, let $\mathcal{Z}=\{\mathbf{n}_1$, $\mathbf{n}_2,\ldots$,
$\mathbf{n}_l\}$ be a set of these vectors. If the vectors $\mathbf{n}_1$,
$\mathbf{n}_2,\ldots$, $\mathbf{n}_l$ all have the same norm $n$, then
we write $\|\mathcal{Z}\|=n$. In this case the set $\mathcal{Z}$ is
called \textit{normed set}.

The vector $|\mathbf{n}|=\left(|n^{(1)}|, |n^{(2)}|\right)$ is called \textit{module} of the vector
$\mathbf{n}$. An important example of a normed set $\mathcal{Z}$ is the set
$$
\mathcal{X}(\mathbf{n})=\{\mathbf{m}: \mathbf{m}=|\mathbf{n}|\}.
$$

For instance, for vectors $(1,2)$ and $(0,1)$ we
have $\mathcal{X}\big((1,2)\big)$ $=\{(1,2)$, $(-1,2)$, $(1,-2)$, $(-1,-2)\}$, and
$\mathcal{X}\big((0,1)\big)$ $=\{(0,1)$, $(0,-1)\}$. Two other important examples of normed sets $\mathcal{Z}$ are the
set of all vectors in $\mathbb{Z}^2$ with norm $n$ and the set of
all vectors in $\mathbb{Z}^2_+$ with norm $n$. These sets are
denoted by $\mathcal{N}(n)$ and $\mathcal{N}^+(n)$, respectively.

Let $\mathbf{n}$ be a vector. A vector $\mathbf{m}$ is said to be a neighbor of vector $\mathbf{n}$ if
$$ \|\mathbf{m}-\mathbf{n}\|=1.$$
If, in addition, $\|\mathbf{n}\|=\|\mathbf{m}\|+1$, then the vector $\mathbf{m}$ is said to be \textit{lower} neighbor of $\mathbf{n}$. Otherwise, if $\|\mathbf{n}\|=\|\mathbf{m}\|-1$, then the vector $\mathbf{m}$ is called \textit{upper} neighbor. The set of all lower neighbors of the vector $\mathbf{n}$ is denoted $\mathcal{M}^-(\mathbf{n})$, and the set of all upper neighbors of the vector $\mathbf{n}$ is denoted $\mathcal{M}^+(\mathbf{n})$. Let $\mathbf{n}=\left(n^{(1)}, n^{(2)}\right)$. The
number of zero components of this vector is denoted by
$d_0(\mathbf{n})$. That is, $d_0(\mathbf{n})=1$ if either $n^{(1)}=0$, $n^{(2)}\neq0$  or $n^{(2)}=0$, $n^{(1)}\neq0$.

The total number of vectors in the set
$\mathcal{M}^+(\mathbf{n})\cup\mathcal{M}^-(\mathbf{n})$ is the same
for all $\mathbf{n}\in\mathbb{Z}^2$ and equal to four. The total number of vectors
in the sets $\mathcal{M}^-(\mathbf{n})$ and
$\mathcal{M}^+(\mathbf{n})$ are $d-d_0(\mathbf{n})$ and
$d+d_0(\mathbf{n})$, respectively. In Figure 1, we illustrate lower and upper neighbor vectors in two situations. In one situation, the vectors $\mathbf{k}$, $\mathbf{l}$ and $\mathbf{m}$ are upper neighbors of the vector $\mathbf{n}$ the total number of which is three, while there is only the single lower neighbor vector $\mathbf{p}$. The vector $\mathbf{n}$ is located on the boundary, so $d_0(\mathbf{n})=1$. In the other situation where the vector $\boldsymbol{c}$ is in interior,  the vectors $\mathbf{a}$ and $\mathbf{b}$ are upper neighbors of the vector $\mathbf{c}$. The two other vectors $\mathbf{d}$ and $\mathbf{e}$ are its lower neighbors.

\begin{figure}
\includegraphics[width=10cm, height=12cm]{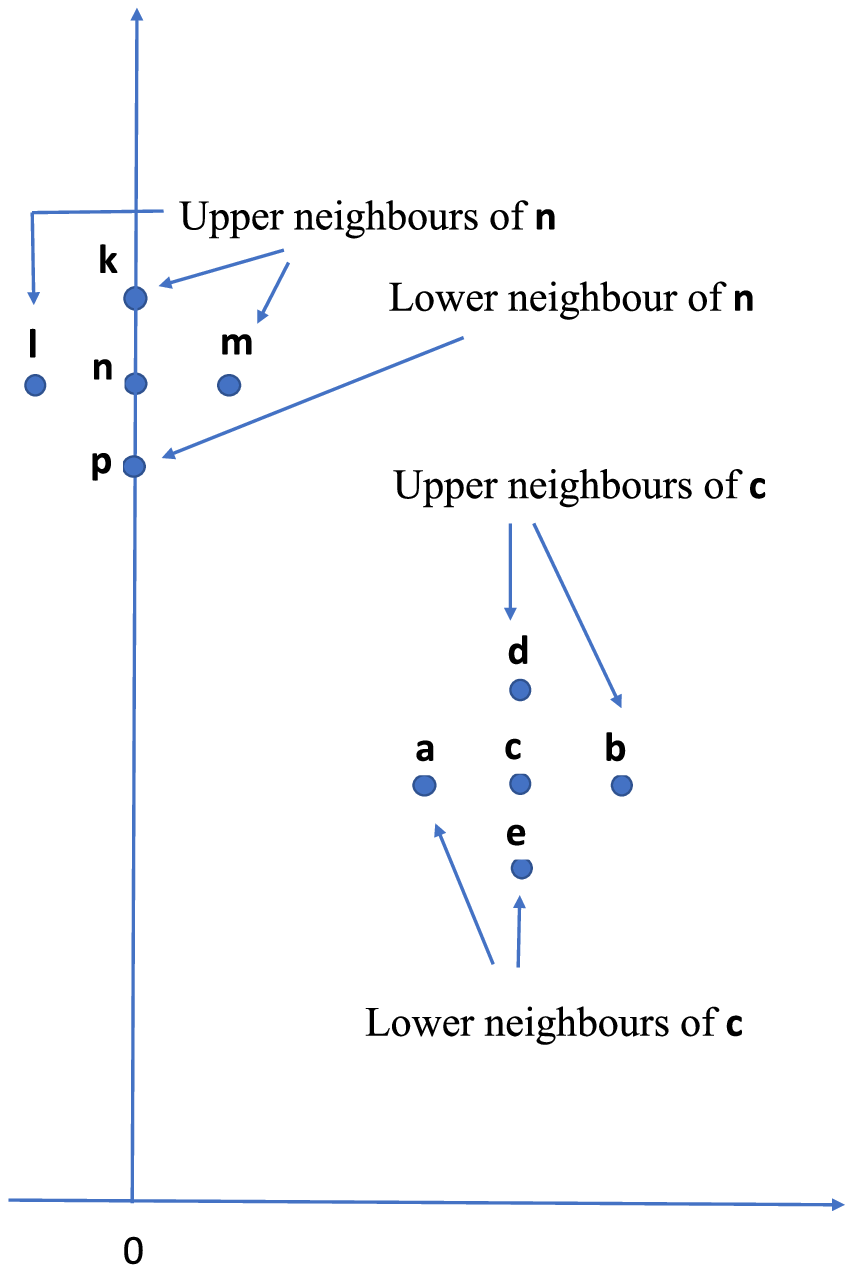}
\caption{Illustration of the lower and upper sets of vectors.}
\end{figure}

The unit vector $(1,1)$ is denoted by $\mathbf{1}$. Writing $\mathbf{n}\geq\mathbf{m}$ means that $n^{(i)}\geq m^{(i)}$
for $i=1,2$, e.g. $\mathbf{n}\geq\mathbf{1}$ means
$n^{(1)}\geq 1$ and $n^{(2)}\geq 1$; $\mathbf{n}\neq\mathbf{0}$ means that
at least one of $n^{(1)}$, $n^{(2)}$ differs from $0$. The following additional widely known conventions are as follows: $\binom{0}{0}=1$ and for $n<k$, $\binom{n}{k}=0$.

\begin{defn}\label{d3} Let
$\mathcal{Z}=\{\mathbf{n}_1$, $\mathbf{n}_2,\ldots$,
$\mathbf{n}_l\}$ be a normed set. The random
variable $f(\mathcal{Z})$ is called \textit{the number of crossings}
$\mathcal{Z}$, if there exist time instants
$0<t_1(\mathcal{Z})<t_2(\mathcal{Z})<\ldots<t_{f(\mathcal{Z})}(\mathcal{Z})<\tau$
such that $\mathbf{S}_{t_i(\mathcal{Z})}\in\mathcal{Z}$,
$i=1,2,\ldots,{f}(\mathcal{Z})$. If the set $\mathcal{Z}$ contains the
only single element $\mathbf{n}$, then the notation
$f(\mathbf{n})$ is used, and $f(\mathbf{n})$ is called \textit{the
number of state-crossings} $\mathbf{n}$.
\end{defn}

\begin{rem}It follows from Definition \ref{d3} that $f(\mathcal{Z})=\sum_{i=1}^{l}f(\mathbf{n}_i)$. That is, $f(\mathcal{Z})$ characterizes the total number of crossing the states $\mathbf{n}_1$, $\mathbf{n}_2,\ldots$, $\mathbf{n}_l$ belonging to $\mathcal{Z}$.
\end{rem}

\begin{rem}
Definition \ref{d3} does not explicitly use the norm of the set. Nevertheless, calling a set normed is important in the definition, since otherwise the definition becomes misleading.
In the following two definitions the norm of the set is used explicitly.
\end{rem}

\begin{defn}\label{d4} For a normed set
$\mathcal{Z}=\{\mathbf{n}_1$, $\mathbf{n}_2,\ldots$,
$\mathbf{n}_l\}$, $\|\mathcal{Z}\|\neq0$, the random
variable $\overrightarrow{f}(\mathcal{Z})$ is called the number of
\textit{up-directed} crossings $\mathcal{Z}$, if there exist
time instants $0< t_1(\mathcal{Z})<t_2(\mathcal{Z})<\ldots<
t_{\overrightarrow{f}(\mathcal{Z})}(\mathcal{Z})<\tau$ such that
$\mathbf{S}_{t_i(\mathcal{Z})}\in\mathcal{Z}$ and
$\|\mathbf{S}_{t_{i}(\mathcal{Z})-1}\|=\|\mathcal{Z}\|-1$,
$i=1,2,\ldots,\overrightarrow{f}(\mathcal{Z})$. If the set
$\mathcal{Z}$ contains the only single element $\mathbf{n}$, then the
notation $\overrightarrow{f}(\mathbf{n})$ is used, and
$\overrightarrow{f}(\mathbf{n})$ is called \textit{the number of
up-directed state-crossings} $\mathbf{n}$.
\end{defn}

\begin{defn}\label{d5}
For any normed set $\mathcal{Z}=\{\mathbf{n}_1$, $\mathbf{n}_2,\ldots$,
$\mathbf{n}_l\}$, the random
variable $\overleftarrow{f}(\mathcal{Z})$ is called the number of
\textit{down-directed} crossings the set $\mathcal{Z}$, if there
exist time instants  $0<
t_1(\mathcal{Z})<t_2(\mathcal{Z})<\ldots<
t_{\overleftarrow{f}(\mathcal{Z})}(\mathcal{Z})<\tau$ such that
$\mathbf{S}_{t_i(\mathcal{Z})}\in\mathcal{Z}$ and
$\|\mathbf{S}_{t_{i}(\mathcal{Z})-1}\|=\|\mathcal{Z}\|+1$,
$i=1,2,\ldots,\overleftarrow{f}(\mathcal{Z})$. If the set
$\mathcal{Z}$ contains the only single element $\mathbf{n}$, then the
notation $\overleftarrow{f}(\mathbf{n})$ is used, and
$\overleftarrow{f}(\mathbf{n})$ is called \textit{the number of
down-directed state-crossings} $\mathbf{n}$.
\end{defn}

\begin{rem}
Unlike Definition \ref{d4}, Definition \ref{d5}
implies that $\mathcal{Z}$ can be equal to $\{\mathbf{0}\}$, and
$\overleftarrow{f}(\mathbf{0})$ is defined. Notice that $f(\mathbf{0})=\overleftarrow{f}(\mathbf{0})$.
\end{rem}

\begin{rem}
For one-dimensional random walk in $\mathbb{Z}^1$ we use the scalar
notation. The random walk is denoted by $S_t$ and the numbers of
state-crossings (or, more exactly, level-crossings) are denoted by
$f(n)$, $\overrightarrow{f}(n)$ or $\overleftarrow{f}(n)$ for the
cases of undirected, up-directed and down-directed level-crossings,
respectively.
\end{rem}

\begin{defn}
Let $\mathcal{M}^+(\mathbf{n})$ be the upper set of the vector
$\mathbf{n}$, and let $\mathbf{m}_1$,
$\mathbf{m}_2$,\ldots,$\mathbf{m}_k$ be the elements of $\mathcal{M}^+(\mathbf{n})$,
where $k=2+d_0(\mathbf{n})$ is the number of its elements. The new
set containing the elements $|\mathbf{m}_1|$,
$|\mathbf{m}_2|$,\ldots,$|\mathbf{m}_k|$ is denoted by
$|\mathcal{M}^+(\mathbf{n})|$ and called \textit{positive upper set}.
\end{defn}

\begin{ex}\label{ex1}
For the vector $(-1, 0)$, we have $\mathcal{M}^+\big((-1, 0)\big)$ $=\{(-2, 0)$, $(-1, 1)$, $(-1, -1)\}$. Hence,
$$
|\mathcal{M}^+\big((-1,0)\big)|=\{(2,0), (1,1)\}.
$$
\end{ex}

Notice that while the set $\mathcal{M}^+\big((-1,0)\big)$ contains 3
elements, the number of elements in the set
$|\mathcal{M}^+\big((-1,0)\big)|$ is only 2. This is because
$|(-1,1)|=|(-1,-1)|=(1,1)$. In this case, we say that the
rank of element $(1,1)\in|\mathcal{M}^+\big((-1,0)\big)|$ is 2.

\smallskip
Below, the general definition of the rank of elements in
$|\mathcal{M}^+(\mathbf{n})|$,
$\mathbf{n}\in\mathbb{Z}^2\setminus\{\mathbf{0}\}$, is provided.

\begin{defn}\label{d2}
We say that the element $\mathbf{m}\in|\mathcal{M}^+(\mathbf{n})|$
has rank $2$ if there are two distinct elements in
$\mathcal{M}^+(\mathbf{n})$ denoted by $\mathbf{m}_1$ and
$\mathbf{m}_2$ such that $|\mathbf{m}_1|=|\mathbf{m}_2|=\mathbf{m}$.
If the set $\mathcal{M}^+(\mathbf{n})$ contains only a single
element $\mathbf{m}_1$ such that $|\mathbf{m}_1|=\mathbf{m}$, then
the rank of the element $\mathbf{m}$ is $1$. The rank of the element
$\mathbf{m}$ will be denoted by $r(\mathbf{m})$.
\end{defn}

\begin{figure}
\includegraphics[width=10cm, height=12cm]{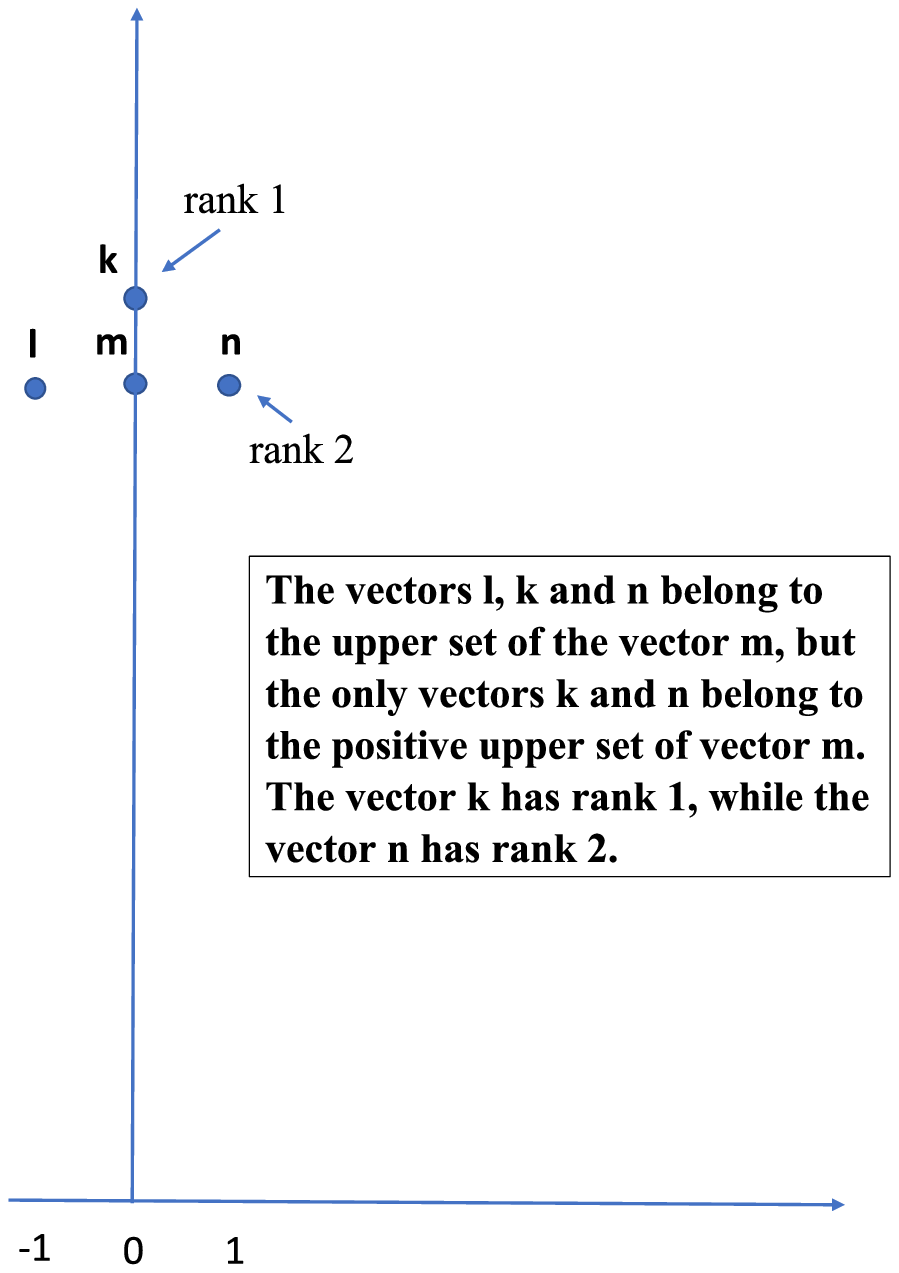}
\caption{Illustration of the set $|\mathcal{M}^+(\mathbf{n})|$ and ranks of states from Definition \ref{d2}.}
\end{figure}

The property from Definition \ref{d2} is shown in Figure 2. The points $\mathbf{k}$, $\mathbf{l}$ and $\mathbf{n}$ all belong to $\mathcal{M}^+(\mathbf{m})$. Then, the only points $\mathbf{k}$ and $\mathbf{n}$ belong to $|\mathcal{M}^+(\mathbf{m})|$. The rank of the point $\mathbf{k}$ is 1, and the rank of the point $\mathbf{n}$ is 2 because of symmetric reflection from the boundary.

\section{Main results}\label{S3}

We need to first define the random objects that are used in the formulation of the basic theorems. These random objects are the Markov fields $P_{\mathbf{n}}$ and $Q_{\mathbf{n}}$ and Markov chain $R_n$  ($n>0$). The random objects $P_{\mathbf{n}}$ and $Q_{\mathbf{n}}$ are required to describe the probability distributions of the number of crossings the state $\mathbf{n}$ and sets of states $\mathcal{X}(\mathbf{n})$, respectively. The Markov
chain $R_n$ is required to describe the distribution of the number
of crossings the set of states $\mathcal{N}(n)$. The random field $P_{\mathbf{n}}$ is defined on $\mathbb{Z}^2$, while the random field $Q_{\mathbf{n}}$ is defined on $\mathbb{Z}_+^2$.

\subsection{The Markov field $P_{\mathbf{n}}$}\label{S3.1}
The random field $P_{\mathbf{n}}$, $\mathbf{n}\in\mathbb{Z}^2$, is
defined as follows.
For any set $\mathcal{Z}\subset\mathbb{Z}^2$
denote by $\mathcal{P}(\mathcal{Z})$ the filtration of the set
$\{P_{\mathbf{m}}: \mathbf{m}\in\mathcal{Z}\}$. Then, with fixed
$P_{\mathbf{0}}=1$ for any
$\mathbf{n}\in\mathbb{Z}^2\setminus\{\mathbf{0}\}$, and all
$k=0,1,\ldots,$ the Markov property
\begin{equation}\label{eq-3.10}
\mathsf{P}\{P_{\mathbf{n}}=k~|~\mathcal{P}\big(\mathbb{Z}^2\setminus\{\mathbf{n}\}\big)\}=
\mathsf{P}\{P_{\mathbf{n}}=k~|~\mathcal{P}
\big(\mathcal{M}^-(\mathbf{n})\cup\mathcal{M}^+(\mathbf{n})\big)\}
\end{equation}
is assumed to be satisfied and, hence, the field $P_{\mathbf{n}}$ is
Markov. (Recall that the set
$\mathcal{M}^-(\mathbf{n})\cup\mathcal{M}^+(\mathbf{n})$ is the set
of all neighbor elements of the vector $\mathbf{n}$.)
Denote by
$p(\mathbf{m},\mathbf{n})$ the number of immediate (one-step)
transitions from the state $\mathbf{m}$ to the state $\mathbf{n}$, where $\mathbf{m}\in\mathcal{M}^-(\mathbf{n})\cup\mathcal{M}^+(\mathbf{n})$. So,
\begin{equation}\label{eq-3.11}
P_{\mathbf{n}}=\sum_{\mathbf{m}\in\mathcal{M}^-(\mathbf{n})\cup\mathcal{M}^+(\mathbf{n})}
p(\mathbf{m},\mathbf{n}).
\end{equation}
Then the physical meaning of $P_{\mathbf{n}}$ is the total number of transitions
to the state $\mathbf{n}$. In agreement with \eqref{eq-3.10} and \eqref{eq-3.11} the transitions $p(\mathbf{m},\mathbf{n})$ must be Markov, since any transition $p(\mathbf{m},\mathbf{n})$ depends only on the transitions to the state $\mathbf{m}$ from its neighbor states and does not depend on the states outside of $\mathcal{M}^-(\mathbf{m})\cup\mathcal{M}^+(\mathbf{m})$ .  For instance,
\begin{equation}\label{eq-3.14}
\begin{aligned}
&\mathsf{P}\{p(\mathbf{m},\mathbf{n})=k~|~p(\mathbf{m}^*,\mathbf{m})=l\}\\
&=
\mathsf{P}\{p(\mathbf{m},\mathbf{n})=k~|~p(\mathbf{m}^*,\mathbf{m})=l, p(\mathbf{m}_1,\mathbf{n}_1), p(\mathbf{m}_2,\mathbf{n}_2),\ldots\},
\end{aligned}
\end{equation}
where $\mathbf{m}^*\in\mathcal{M}^-(\mathbf{m})\cup\mathcal{M}^+(\mathbf{m})$, and $\mathbf{m}_i\notin\mathcal{M}^-(\mathbf{m})\cup\mathcal{M}^+(\mathbf{m})$ and $\|\mathbf{m}_i-\mathbf{n}_i\|=1$, $i=1,2,\ldots$ .

\smallskip
We have the following elementary properties of $p(\mathbf{m},\mathbf{n})$ and $P_{\mathbf{n}}$.
\begin{enumerate}
  \item $P_{\mathbf{0}}=1$.
  \smallskip
  \item $\mathsf{P}\{p(\mathbf{0},\mathbf{1}_1)=1\}=\mathsf{P}\{p(\mathbf{0},\mathbf{1}_2)=1\}=\mathsf{P}\{p(\mathbf{0},-\mathbf{1}_1)=1\}=\mathsf{P}\{p(\mathbf{0},-\mathbf{1}_2)=1\}
=1/4$, and
$ p(\mathbf{0},\mathbf{1}_1)+p(\mathbf{0},\mathbf{1}_2)+p(\mathbf{0},-\mathbf{1}_1)+p(\mathbf{0},-\mathbf{1}_2)=1.$
  \smallskip
  \item $\mathsf{P}\{p(\mathbf{1}_1,\mathbf{0})=1\}=\mathsf{P}\{p(\mathbf{1}_2,\mathbf{0})=1\}=\mathsf{P}\{p(-\mathbf{1}_1,\mathbf{0})=1\}=\mathsf{P}\{p(-\mathbf{1}_2,\mathbf{0})=1\}
=1/4$, and
$ p(\mathbf{1}_1,\mathbf{0})+p(\mathbf{1}_2,\mathbf{0})+p(-\mathbf{1}_1,\mathbf{0})+p(-\mathbf{1}_2,\mathbf{0})=1.$
  \smallskip
  \item if $\mathbf{m}\in\mathcal{M}^-(\mathbf{n})$, $\|\mathbf{n}\|\geq2$, then
  $$
  \mathsf{P}\left\{p(\mathbf{m},\mathbf{n})=n~\Big|~\sum_{\mathbf{m}^*\in\mathcal{M}^-(\mathbf{m})}p(\mathbf{m}^*,\mathbf{m})=k\right\}=\binom{k+n-1}{n}\left(\frac{1}{4}\right)^n\left(\frac{3}{4}\right)^k.
  $$
  Note, that under the condition $\left\{\sum_{\mathbf{m}^*\in\mathcal{M}^-(\mathbf{m})}p(\mathbf{m}^*,\mathbf{m})=1\right\}$ the conditional distribution is geometric, while under the more general condition $\left\{\sum_{\mathbf{m}^*\in\mathcal{M}^-(\mathbf{m})}p(\mathbf{m}^*,\mathbf{m})=k\right\}$ it is negative binomial. This means that the condition ``generates" $k$ independent geometrically distributed random variables, and the conditional distribution is the distribution of the sum of independent geometrically distributed random variables that yields negative binomial distribution. This property is widely used in the paper.
  \smallskip
  \item if $\mathbf{m}\in\mathcal{M}^+(\mathbf{n})$, $\mathbf{n}\neq\mathbf{0}$, then
  $$
  \mathsf{P}\left\{p(\mathbf{m},\mathbf{n})=n~\Big|~\sum_{\mathbf{m}^*\in\mathcal{M}^+(\mathbf{m})}p(\mathbf{m}^*,\mathbf{m})=k\right\}=\binom{k+n-1}{n}\left(\frac{1}{4}\right)^n\left(\frac{3}{4}\right)^k.
  $$
\end{enumerate}

\subsection{The Markov field $Q_{\mathbf{n}}$}\label{S3.2}
The random field $Q_{\mathbf{n}}$, $\mathbf{n}\in\mathbb{Z}_+^2$, is
defined as follows.
For any set $\mathcal{Z}\subset\mathbb{Z}_+^2$
denote by $\mathcal{Q}(\mathcal{Z})$ the filtration of the set
$\{Q_{\mathbf{m}}: \mathbf{m}\in\mathcal{Z}\}$. Then, with fixed
$Q_{\mathbf{0}}=1$ for any
$\mathbf{n}\in\mathbb{Z}^2_+\setminus\{\mathbf{0}\}$, and all
$k=0,1,\ldots,$ the Markov property
\begin{equation}\label{eq-3.12}
\mathsf{P}\{Q_{\mathbf{n}}=k~|~\mathcal{Q}\big(\mathbb{Z}_+^2\setminus\{\mathbf{n}\}\big)\}=
\mathsf{P}\{Q_{\mathbf{n}}=k~|~\mathcal{Q}
\big(\mathcal{M}^-(\mathbf{n})\cup\mathcal{M}^+(\mathbf{n})\big)\}
\end{equation}
is assumed to be satisfied and, hence, the field $Q_{\mathbf{n}}$ is
Markov.
Denote by
$q(\mathbf{m},\mathbf{n})$ the number of immediate (one-step)
transitions from the state $\mathbf{m}$ to the state $\mathbf{n}$, where $\mathbf{m}\in\mathcal{M}^-(\mathbf{n})\cup\mathcal{M}^+(\mathbf{n})$. So,
\begin{equation}\label{eq-3.13}
Q_{\mathbf{n}}=\sum_{\mathbf{m}\in\mathcal{M}^-(\mathbf{n})\cup\mathcal{M}^+(\mathbf{n})}
q(\mathbf{m},\mathbf{n}).
\end{equation}
So, similarly to $P_{\mathbf{n}}$, the meaning of $Q_{\mathbf{n}}$ is the total number of transitions
to the state $\mathbf{n}$. In agreement with \eqref{eq-3.12} and \eqref{eq-3.13} the transitions $q(\mathbf{m},\mathbf{n})$ must be Markov, and similarly to \eqref{eq-3.14}
$$
\begin{aligned}
&\mathsf{P}\{q(\mathbf{m},\mathbf{n})=k~|~q(\mathbf{m}^*,\mathbf{m})=l\}\\
&=
\mathsf{P}\{q(\mathbf{m},\mathbf{n})=k~|~q(\mathbf{m}^*,\mathbf{m})=l, q(\mathbf{m}_1,\mathbf{n}_1), q(\mathbf{m}_2,\mathbf{n}_2),\ldots\},
\end{aligned}
$$
where $\mathbf{m}^*\in\mathcal{M}^-(\mathbf{m})\cup\mathcal{M}^+(\mathbf{m})$, and $\mathbf{m}_i\notin\mathcal{M}^-(\mathbf{m})\cup\mathcal{M}^+(\mathbf{m})$ and $\|\mathbf{m}_i-\mathbf{n}_i\|=1$, $i=1,2,\ldots$ .

\smallskip
We have the following elementary properties of $q(\mathbf{m},\mathbf{n})$ and $Q_{\mathbf{n}}$.
\begin{enumerate}
  \item $Q_{\mathbf{0}}=1$.
  \smallskip
  \item $\mathsf{P}\{q(\mathbf{0},\mathbf{1}_1)=1\}=\mathsf{P}\{q(\mathbf{0},\mathbf{1}_2)=1\}
=1/2$, and
$ q(\mathbf{0},\mathbf{1}_1)+q(\mathbf{0},\mathbf{1}_2)=1.$
  \smallskip
  \item $\mathsf{P}\{q(\mathbf{1}_1,\mathbf{0})=1\}=\mathsf{P}\{q(\mathbf{1}_2,\mathbf{0})=1\}
=1/2$, and
$ q(\mathbf{1}_1,\mathbf{0})+q(\mathbf{1}_2,\mathbf{0})=1.$
  \smallskip

  \item if $\mathbf{m}\in|\mathcal{M}^+(\mathbf{n})|$ and $\mathbf{n}\neq\mathbf{0}$, then
  $$
  \begin{aligned}
  &\mathsf{P}\left\{q(\mathbf{m},\mathbf{n})=n~\Big|~\sum_{\mathbf{m}^*\in|\mathcal{M}^+(\mathbf{m})|}q(\mathbf{m}^*,\mathbf{m})=k\right\}\\
  &=\binom{k+n-1}{n}\left(\frac{1}{4}\right)^n\left(\frac{3}{4}\right)^k.
  \end{aligned}
  $$
  \smallskip
  \item if $\mathbf{m}\in\mathcal{M}^-(\mathbf{n})$, $d_0(\mathbf{m})=1$ and $\|\mathbf{n}\|\geq2$,  then
  $$
  \begin{aligned}
  &\mathsf{P}\left\{q(\mathbf{m},\mathbf{n})=n~\Big|~\sum_{\mathbf{m}^*\in\mathcal{M}^-(\mathbf{m})}q(\mathbf{m}^*,\mathbf{m})=k\right\}\\
  &=\binom{k+n-1}{n}\left(\frac{r(\mathbf{n})}{4}\right)^n
  \left(\frac{4-r(\mathbf{n})}{4}\right)^k,
  \end{aligned}
  $$
where $r(\mathbf{n})$ is the rank of the vector $\mathbf{n}\in|\mathcal{M}^+(\mathbf{m})|$ (see Definition \ref{d2}).
  \smallskip
  \item if $\mathbf{m}\in\mathcal{M}^-(\mathbf{n})$, $d_0(\mathbf{m})=0$ and $\|\mathbf{n}\|\geq2$, then
  $$
  \begin{aligned}
  &\mathsf{P}\left\{q(\mathbf{m},\mathbf{n})=n~\Big|~\sum_{\mathbf{m}^*\in\mathcal{M}^-(\mathbf{m})}q(\mathbf{m}^*,\mathbf{m})=k\right\}\\
  &=\binom{k+n-1}{n}\left(\frac{1}{4}\right)^n\left(\frac{3}{4}\right)^k.
  \end{aligned}
  $$
\end{enumerate}
\medskip

Property (5) for $\mathsf{P}\big\{q(\mathbf{m},\mathbf{n})=n~\big|~\sum_{\mathbf{m}^*\in\mathcal{M}^-(\mathbf{m})}q(\mathbf{m^*},\mathbf{m})=k\big\}$ is seen from Figure 3. When the vector $\mathbf{m}$ belongs to the boundary and $\mathbf{n}$ is out of the boundary, the probability $\mathsf{P}\big\{q(\mathbf{m},\mathbf{n})=1~\big|~\sum_{\mathbf{m}^*\in\mathcal{M}^-(\mathbf{m})}q(\mathbf{m^*},\mathbf{m})=1\big\}$ is 1/4 rather than 3/16. This is due to the reflection mechanism giving the rate $2\times1/4=1/2$.

\begin{figure}
\includegraphics[width=10cm, height=12cm]{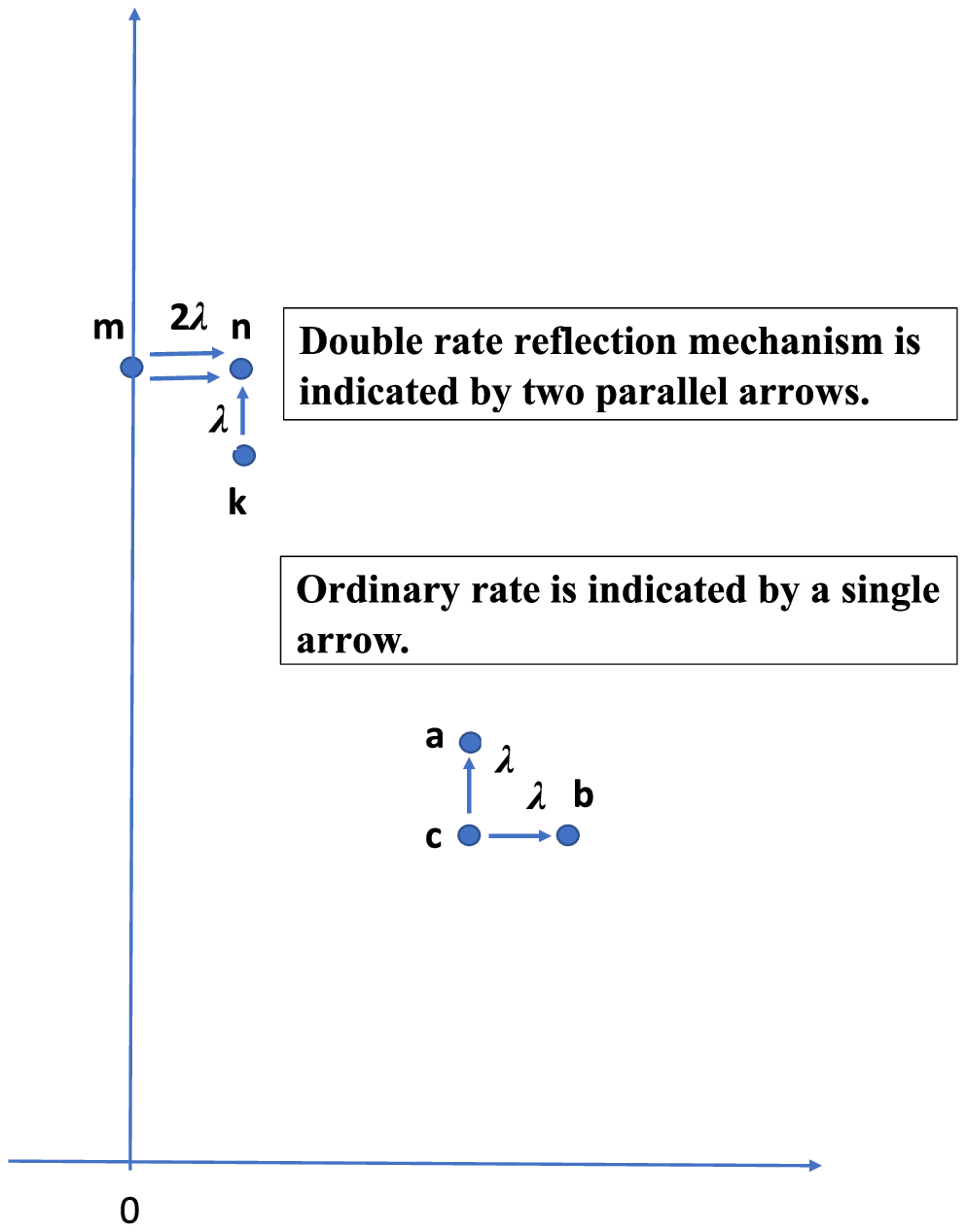}
\caption{Graphical explanation of success probabilities in Properties (4),  (5) and (6) related to the field $Q_{\mathbf{n}}$.}
\end{figure}

\subsection{The Markov chain $R_n$}\label{S3.3}

Discrete time nonnegative integer-valued Markov chain $R_n$, $n\geq1$, satisfies the following properties.
\begin{enumerate}
  \item $R_0=1$.
  \item $\mathsf{P}\{R_{n}=k|R_{n-1}=m\}=\binom{k+m-1}{k}\left(\frac{2n-1}
{4n}\right)^m\left(\frac{2n+1}{4n}\right)^k$.
\end{enumerate}

\subsection{Formulation of the results}\label{S3.4}

The main results of the paper are specified as follows. Theorem
\ref{t2} is the theorem on the probability distribution of the
number of crossings the set $\mathcal{N}(n)$. Theorem \ref{t1} is
the theorem on the probability distributions of the number of
crossings the sets $\mathcal{X}(\mathbf{n})$. Theorem \ref{t3} is
the theorem on the probability distributions of the number of
state-crossings. Theorem \ref{t4} describes the
expectations of crossings the sets of states
$\mathcal{X}(\mathbf{n})$ and states $\mathbf{n}$. Specifically,
relation \eqref{eq-3.9} of the theorem contains the important claim
that the expectation of the number of crossings any
state $\mathbf{n}$ is equal to 1, and relations \eqref{eq-3.7} and \eqref{eq-3.8} nontrivial results on the expected numbers of up- and down-directed crossings of state $\mathbf{n}$.

\begin{thm}\label{t2}
For any integer $n\geq1$ and all $k=0,1,\ldots$,
\begin{equation}\label{eq-2.11}
\mathsf{P}\left\{\overrightarrow{f}[\mathcal{N}(n)]=k\right\}
=\mathsf{P}\left\{R_{n-1}=k
\right\},
\end{equation}
\begin{equation}\label{eq-2.12}
\mathsf{P}\left\{\overleftarrow{f}[\mathcal{N}(n)]=k\right\}
=\mathsf{P}\left\{R_{n}=k
\right\},
\end{equation}
\begin{equation}\label{eq-2.10}
\mathsf{P}\{f[\mathcal{N}(n)]=k\}=\mathsf{P}\left\{R_{n-1}+R_n=k
\right\}.
\end{equation}
\end{thm}

\begin{rem}
Relations \eqref{eq-2.12} and \eqref{eq-2.10} are defined for $n=0$ as well. Specifically,
$$
\mathsf{P}\left\{\overleftarrow{f}[\mathcal{N}(0)]=1\right\}
=\mathsf{P}\left\{f[\mathcal{N}(0)]=1\right\}=\mathsf{P}\left\{\overrightarrow{f}[\mathcal{N}(1)]=1\right\}=
1.
$$
\end{rem}

\begin{thm}\label{t1}
For any $\mathbf{n}\in\mathbb{Z}^2_+\setminus\{\mathbf{0}\}$ and all
$k=0,1,\ldots,$
\begin{equation}\label{eq-3.1}
\begin{aligned}
\mathsf{P}\left\{\overrightarrow{f}[\mathcal{X}(\mathbf{n})]=k\right\}
= \mathsf{P}\left\{\sum_{\mathbf{m}\in\mathcal{M}^-(\mathbf{n})}
q(\mathbf{m},\mathbf{n})=k
\right\},
\end{aligned}
\end{equation}
\begin{equation}\label{eq-3.2}
\begin{aligned}
\mathsf{P}\left\{\overleftarrow{f}[\mathcal{X}(\mathbf{n})]=k\right\}
= \mathsf{P}\left\{\sum_{\mathbf{m}\in|\mathcal{M}^+(\mathbf{n})|}
q(\mathbf{m},\mathbf{n})=k
\right\},
\end{aligned}
\end{equation}
and
\begin{equation}\label{eq-3.3}
\mathsf{P}\left\{{f}[\mathcal{X}(\mathbf{n})]=k\right\}=
\mathsf{P}\left\{Q_{\mathbf{n}}=k
\right\}.
\end{equation}
\end{thm}

\begin{rem}\label{r2}
$$\mathsf{P}\left\{\overleftarrow{f}[\mathcal{X}(\mathbf{0})]=1\right\}
=\mathsf{P}\left\{\overleftarrow{f}(\mathbf{0})=1\right\}=\mathsf{P}\left\{f(\mathbf{0})=1\right\}=1.$$
\end{rem}

\begin{thm}\label{t3}
For any $\mathbf{n}\in\mathbb{Z}^2\setminus\{\mathbf{0}\}$ and all
$k=0,1,\ldots,$
\begin{equation}\label{eq-3.4}
\begin{aligned}
\mathsf{P}\left\{\overrightarrow{f}(\mathbf{n})=k\right\}
= \mathsf{P}\left\{\sum_{\mathbf{m}\in\mathcal{M}^-(\mathbf{n})}
p(\mathbf{m},\mathbf{n})=k
\right\},
\end{aligned}
\end{equation}
\begin{equation}\label{eq-3.5}
\begin{aligned}
\mathsf{P}\left\{\overleftarrow{f}(\mathbf{n})=k\right\}
= \mathsf{P}\left\{\sum_{\mathbf{m}\in\mathcal{M}^+(\mathbf{n})}
p(\mathbf{m},\mathbf{n})=k
\right\},
\end{aligned}
\end{equation}
and
\begin{equation}\label{eq-3.6}
\begin{aligned}
\mathsf{P}\left\{{f}(\mathbf{n})=k\right\}=
\mathsf{P}\left\{P_{\mathbf{n}}=k
\right\}.
\end{aligned}
\end{equation}
\end{thm}

\begin{thm}\label{t4}
For any $\mathbf{n}\in\mathbb{Z}^2\setminus\{\mathbf{0}\}$, we have:
\begin{equation}\label{eq-3.7}
\mathsf{E}\left\{\overrightarrow{f}(\mathbf{n})\right\}=2^
{d_0(\mathbf{n})-2}
\mathsf{E}\left\{\overrightarrow{f}[\mathcal{X}(\mathbf{n})]\right\}
=\frac{2-d_0(\mathbf{n})}{4},
\end{equation}
\begin{equation}\label{eq-3.8}
\mathsf{E}\left\{\overleftarrow{f}(\mathbf{n})\right\}=2^
{d_0(\mathbf{n})-2}
\mathsf{E}\left\{\overleftarrow{f}[\mathcal{X}(\mathbf{n})]\right\}
=\frac{2+d_0(\mathbf{n})}{4},
\end{equation}
and
\begin{equation}\label{eq-3.9}
\mathsf{E}\left\{{f}(\mathbf{n})\right\}=2^
{d_0(\mathbf{n})-2}
\mathsf{E}\left\{{f}[\mathcal{X}(\mathbf{n})]\right\} =1.
\end{equation}
\end{thm}

\section{Level crossings in the one-dimensional random walk}\label{S4}

Let $Z_n$ denote the number of offspring in the $n$th generation
of the Galton-Watson branching process with $Z_0=1$ and the
offspring distribution $\mathsf{P}\{Z_1=k\}=1/2^{k+1}$. We have the following statement.

\begin{prop}\label{c2}
For the one-dimensional random walk $S_t$ the following
results are true. For all $n\neq0$ and $k=1,2,\ldots$
\begin{equation}\label{eq-5.31}
\mathsf{P}\left\{\overrightarrow{f}(n)=k\right\}=\frac{1}{2}\mathsf{P}\{Z_{|n|-1}=k\},
\end{equation}
\begin{equation}\label{eq-5.32}
\mathsf{P}\left\{\overleftarrow{f}(n)=k\right\}=\frac{1}{2}\mathsf{P}\{Z_{|n|}=k\},
\end{equation}
and
\begin{equation}\label{eq-5.33}
\mathsf{P}\left\{f(n)=k\right\}=\frac{1}{2}\mathsf{P}\{Z_{|n|-1}+Z_{|n|}=k\},
\end{equation}
and the probabilities
$\mathsf{P}\left\{\overrightarrow{f}(n)=0\right\}$,
$\mathsf{P}\left\{\overleftarrow{f}(n)=0\right\}$ and
$\mathsf{P}\left\{f(n)=0\right\}$ are determined from the
normalization conditions:
\begin{eqnarray}
  \mathsf{P}\left\{\overrightarrow{f}(n)=0\right\} &=& \frac{1}{2}+\frac{1}{2}\mathsf{P}\{Z_{|n|-1}=0\},\label{eq-5.34} \\
  \mathsf{P}\left\{\overleftarrow{f}(n)=0\right\} &=& \frac{1}{2}+\frac{1}{2}\mathsf{P}\{Z_{|n|}=0\},\label{eq-5.35} \\
  \mathsf{P}\left\{f(n)=0\right\} &=& \frac{1}{2}+\frac{1}{2}\mathsf{P}\{Z_{|n|-1}+Z_{|n|}=0\}.\label{eq-5.37}
\end{eqnarray}
\end{prop}

\begin{proof}
The proof of this theorem is based on the analysis of the random walk $S_t$ from the position of queueing theory (similarly to that was provided in \cite{Abramov_AoP}).
Assume that in each of its states the random walk $S_t$
stays an exponentially distributed time with parameter 1, prior
moving to the next state. Then, the parameter $t$ in $S_t$ means the
$t$th event of the associated Poisson process with rate 1.
The meaning of the random time instant $\tau$ is then the $\tau$th
event of the same Poisson process (with random order
number). This unusual replacement of deterministic $t$ by random is
made for reduction to queueing processes and using the well-known
properties of Poisson process and exponential distribution.

Since the random walk is symmetric, it is clear that the number of
level-crossings of any negative level, say $(-5)$, has the same
distribution as the number of level-crossings of the corresponding
positive level 5. On the other hand, if
$S_{1}=-1$, then the following excursion of the random walk before
reaching the initial point 0 will be in the negative area only, and
otherwise, if $S_{1}=+1$, in the only positive area. This means that the
random walk can be studied in the positive area, starting from
$S_1=1$. That is, for positive $n$, and for all $k=0,1,\ldots$,
\begin{eqnarray}
\mathsf{P}\left\{\overrightarrow{f}(n)=k\right\}&=&\frac{1}{2}
\mathsf{P}\left\{\overrightarrow{f}(n)=k|S_1=1\right\},\label{eq-5.21}\\
\mathsf{P}\left\{\overleftarrow{f}(n)=k\right\}&=&\frac{1}{2}
\mathsf{P}\left\{\overleftarrow{f}(n)=k|S_1=1\right\},\label{eq-5.22}
\end{eqnarray}
and
\begin{equation}\label{eq-5.23}
\mathsf{P}\{f(n)=k\}=\frac{1}{2}\mathsf{P}\{f(n)=k|S_1=1\}.
\end{equation}
Apparently,
$\mathsf{P}\Big\{\overrightarrow{f}(1)=k|S_1=1\Big\}=\delta_{k,1}$,
where $\delta_{k,j}$ is Kronecker's delta.
\smallskip

Consider now the $M/M/1/n$ queueing system (with $n-1$ waiting
places), in which interarrival and service times are exponentially
distributed with same parameter. The value of the parameter does not
matter. However, to make the paths of queueing process and random
walk equivalent, one can reckon that the mean interarrival and
service times are equal to 2. In that case, the times between the
consecutive events (such as arrivals and service completions) are exponentially
distributed with parameter 1.

\begin{lem}\label{l1}
Let $L_n$ denote the number of losses during a busy period. Then,
for all $k=0,1,\ldots$,
\begin{equation}\label{eq-5.27}
\mathsf{P}\{L_n=k\}=\mathsf{P}\left\{\overrightarrow{f}(n+1)=k|S_1=1\right\}
= \mathsf{P}\{Z_{n}=k\}.
\end{equation}
\end{lem}
\begin{proof}
The number of losses $L_n$ in the $M/M/1/n$ queueing system can be
specified as follows. Let $z_j$ denotes the number of times during
the busy period when an arriving customer finds $j$ customers in the
system ($0\leq j\leq n-1$), and let $a_1^{(j)}$,
$a_2^{(j)}$,\ldots,$a_{z_j}^{(j)}$ denote the moments of arrivals
when an arriving customer finds $j$ customers in the system,
$b_1^{(j)}$, $b_2^{(j)}$,\ldots,$b_{z_j}^{(j)}$ denote the moments
of service completions (departures) of the customers after which
there are remain only $j$ customers in the system. Apparently,
$z_0=1$, $a_1^{(0)}$ is the moment when the busy period starts, and
$b_1^{(0)}$ is the moment when the busy period ends. The time
intervals
\begin{equation}\label{eq-5.24}
\big[a_1^{(j)}, b_1^{(j)}\big), \big[a_2^{(j)}, b_2^{(j)}\big),
\ldots, \big[a_{z_j}^{(j)}, b_{z_j}^{(j)}\big)
\end{equation}
are contained in the intervals
\begin{equation}\label{eq-5.25}
\big[a_1^{(j-1)}, b_1^{(j-1)}\big), \big[a_2^{(j-1)},
b_2^{(j-1)}\big), \ldots, \big[a_{z_{j-1}}^{(j-1)},
b_{z_{j-1}}^{(j-1)}\big)
\end{equation}
$(j\geq 1)$. Let us delete the intervals of \eqref{eq-5.24} from
those of \eqref{eq-5.25} and merge the ends. Then, according to the
property of the lack of memory, the number of inserted points in
each of the intervals of \eqref{eq-5.25} coincides in distribution
with the number of arrivals per service time and has geometric
distribution with parameter $1/2$. This enables us to
conclude that $\{z_j\}$ have the structure of the Galton-Watson
branching process, where $z_j$ is the number of offspring in the
$j$th generation with $z_0=1$ and
$\mathsf{P}\{z_1=k\}=\left(\frac{1}{2}\right)^{k+1}$,
$k=0,1,\ldots$. That is, $z_j=Z_j$, ($j=0,1,\ldots,n-1$). Let us
consider the specific intervals
\begin{equation}\label{eq-5.29}
\big[a_1^{(n-1)}, b_1^{(n-1)}\big), \big[a_2^{(n-1)},
b_2^{(n-1)}\big), \ldots, \big[a_{z_{n-1}}^{(n-1)},
b_{z_{n-1}}^{(n-1)}\big).
\end{equation}
The arrival instants $a_1^{(n)}$, $a_2^{(n)}$,\ldots,
$a_{L_n}^{(n)}$ are the instants of losses from the system.
Apparently, $L_n$ must coincide in distribution with $Z_n$, the
number of offspring in the $n$th generation of the branching
process.

Consider now the family of $M/M/1/n$ queueing systems,
$n=1,2,\ldots$. On the basis of the above construction, $L_1$ must
coincide in distribution with $Z_1$, $L_2$ with $Z_2$ and so on, and
the family of queueing processes associated with the $M/M/1/n$
queueing systems together with the $M/M/1$ queueing process can be
considered on the common probability space. The paths of the random
walk are given on the same probability space, and relation
\eqref{eq-5.27} is clear due to the coincidence of the paths of the
$M/M/1$ queueing process and random walk.
\end{proof}

From Relation \eqref{eq-5.27} of Lemma \ref{l1} and \eqref{eq-5.21} we obtain \eqref{eq-5.31}. In order to obtain \eqref{eq-5.32} and
\eqref{eq-5.33}, we are to prove the following lemma.

\begin{lem}\label{l2}
 For $n>0$ the following obvious relations are true with probability
 $1$:
\begin{eqnarray}
\overleftarrow{f}(n)&=&\overrightarrow{f}(n+1),\label{eq-5.36}\\
f(n)&=&\overrightarrow{f}(n)+\overrightarrow{f}(n+1).\label{eq-5.40}
\end{eqnarray}
\end{lem}

\begin{proof}
Indeed, the total number of (undirected) level-crossings of the
level $n$ includes the number of level-crossings of level $n$ from
the below plus those from the above, i.e.
\begin{equation}\label{eq-5.41}
f(n)=\overrightarrow{f}(n)+\overleftarrow{f}(n).
\end{equation}
Apparently, $\overleftarrow{f}(n)$ coincides with
$\overrightarrow{f}(n+1)$, since the number of down-crossings from
$n+1$ to $n$ is equal to the number of up-crossings from $n$ to
$n+1$, and \eqref{eq-5.36} follows. Its substitution for
\eqref{eq-5.41} yields \eqref{eq-5.40}.
\end{proof}

\medskip
\noindent
\textit{Continuation of the proof of Proposition \ref{c2}.}
Now, Relation \eqref{eq-5.32} follows from \eqref{eq-5.36} and \eqref{eq-5.22}, and Relation \eqref{eq-5.33} follows from \eqref{eq-5.40} and \eqref{eq-5.23}. From the normalization condition, for $\mathsf{P}\left\{\overrightarrow{f}(n)=0\right\}$ we obtain:
$$
\begin{aligned}
\mathsf{P}\left\{\overrightarrow{f}(n)=0\right\}&=1-\sum_{k=1}^{\infty}\mathsf{P}\left\{\overrightarrow{f}(n)=k\right\}\\
&=1-\frac{1}{2}\sum_{k=1}^{\infty}\mathsf{P}\{Z_{|n|-1}=k\}\\
&=1-\frac{1}{2}(1-\mathsf{P}\{Z_{|n|-1}=0\})\\
&=\frac{1}{2}+\frac{1}{2}\mathsf{P}\{Z_{|n|-1}=0\},
\end{aligned}
$$
and \eqref{eq-5.34} is proved. The derivation of \eqref{eq-5.35} and \eqref{eq-5.37} is similar. Proposition \ref{c2} is proved.
\end{proof}

\section{Proof of Theorem \ref{t2}}\label{S5}
\subsection{Prelude}
In the proof of this and the following theorems, it is assumed that in each of its states the random walk stays an exponentially distributed time with parameter 1, prior moving to
the next state. This change of time is similar to that made in Section \ref{S4} and was used in \cite{Abramov_AoP} in more general situation.
In the case of the
present study, the change of states in the two-dimensional random walk is
associated with four independent and identical Poisson processes
with rate $1/4$ each. Specifically, marking them
$$
\underbrace{a_{1,1}}_{\mathbf{1}_1}, \underbrace{a_{1,2}}_{-\mathbf{1}_1}, \underbrace{a_{2,1}}_{\mathbf{1}_2}, \underbrace{a_{2,2}}_{-\mathbf{1}_2}
$$
indicates that the Poisson process $a_{i,1}$ is associated with the
direction $\mathbf{1}_i$, and the Poisson process $a_{i,2}$ is
associated with the direction $(-\mathbf{1}_i)$, $i=1,2$.
Then, the direction of the random walk is associated with the
minimum of four exponentially distributed ``inter-jump" times.

This construction enables us to model the two-dimensional random walk \textit{in the main quarter plane} as two independent and identical queueing systems.
Study of the random walk in the main quarter plane helps us to model the required characteristics $\overrightarrow{f}[\mathcal{N}(n)]$, $\overleftarrow{f}[\mathcal{N}(n)]$ and $f[\mathcal{N}(n)]$, and reduction to continuous time processes essentially simplify
the analysis, which is based on well-known and elementary results of
continuous Markov chains. Then, the parameter $t$ in $\mathbf{S}_t$
means the $t$th event of the associated Poisson process with rate 1.
The meaning of the random time instant $\tau$ is then the $\tau$th
event of the same Poisson process. In the sequel, any phrase like \textit{time
moment $x$} means that it is spoken about the $x$th event of the
Poisson process.

\subsection{Description of queueing systems and the final measures}\label{S5.2} This section is a simplified version of the corresponding place in \cite{Abramov_AoP}, where more general queueing systems have been considered. The description of the queueing model for the random walk in $\mathbb{Z}^d$ is provided in the Appendix and reprodiced here in the case of random walk in $\mathbb{Z}^2$ for the convenience of reading. Each of the two identical queueing systems is described as follows. Arrivals to each queueing system are Poisson with rate 1/4, and service times are exponentially distributed with parameter 1/4. The aforementioned value of arrival rate 1/4 is not important, and later in  the paper this value  will be denoted by $\lambda$, as acceptable. Since the mean service and mean interarrival time are equal, the notation $\lambda$ will be used for the reciprocal of the mean service time as well. If a system
becomes free, it is switched for a special service with the same mean. This service is
\textit{negative}, and it results in a new customer in the queue. If during a negative service a
new arrival occurs, the negative service remains unfinished and not resumed. The negative service models the reflection at zero and in fact implies the state dependent
arrival rate, which becomes equal to $2\times(1/4)=1/2$ at the moment when the system is
empty. In terms of random walks, it is the situation, when an original (not reflected) one-dimensional random
walk reaches zero at some time moment $s$, and at the next time moment $s + 1$ it must
take one of the values $\pm1$, that corresponds to value $+1$ for an one-dimensional random
walk reflected at zero.

In \cite{Abramov_AoP}, p. 1908, the transition probability from the set of states $\mathcal{N}(n)$ to the set of states $\mathcal{N}(n+1)$ has been derived. In the case of the two-dimensional random walk this probability $p_n(2)$ (the notation is from \cite{Abramov_AoP}) is
$$
p_n(2)=\frac{2n+1}{4n},
$$
and the transition probability from the set of states $\mathcal{N}(n)$ to the set of states $\mathcal{N}(n-1)$ is
$$
q_n(2)=1-p_n(2)=\frac{2n-1}{4n}.
$$
The transition probabilities $p_n(2)$ and $q_n(2)$ coincide with the relative rates of the birth-and-death process $BD(2,2)$
introduced in \cite{Abramov_AoP}:
$$
p_n(2)=\frac{\lambda_n(2,2)}{\lambda_n(2,2)+\mu_n(2,2)}
$$
The details of the derivation of the formula for $p_n(d)$ is provided in the Appendix, and the formula for $p_n(2)$ is a particular case of that formula.

\subsection{Basic lemma and its proof}
The rest of the proof of the theorem is based on the lemma, which is an extension of Lemma \ref{l1}. The proof of this lemma is fully similar to the proof of aforemention Lemma \ref{l1} and reproduced here without specific details.

\begin{lem}\label{l3}
Let $\phi_t$ be a recurrent birth-and-death process with the parameters of birth
and death $\lambda_n$ and $\mu_n$, respectively. Assume that
$\phi_0=1$, and let $\upsilon$ be the time moment of extinction of
the birth-and-death process. Let $g(n)$, $n\geq1$, denote the number of
times during the random interval $[0,\upsilon)$ when immediately
before the time of a birth there become $n$ individuals in the
population, $g(0)=1$. Then,
\begin{equation}\label{eq-5.3}
\mathsf{P}\{g(n)=k\}=\mathsf{P}\left\{V_n=k
\right\},
\end{equation}
where $V_n$ is the positive integer-valued Markov chain satisfying
the following properties:
\begin{eqnarray}
V_0&=&1,\label{eq-5.4}\\
\mathsf{P}\{V_{j+1}=k|V_j=m\}&=&\binom{m+k-1}{k}\left(\frac{\lambda_j}{\lambda_j+\mu_j}\right)^k\left(\frac{\mu_j}{\lambda_j+\mu_j}\right)^m,\label{eq-5.5}\\
j=0,1,\ldots.&&\nonumber
\end{eqnarray}

\end{lem}

\begin{proof}
Let $a_1(j)$, $a_2(j)$,\ldots, $a_{g(j)}(j)$ denote the birth times
immediately before there become $j$ individuals in the population,
and let $b_1(j)$, $b_2(j)$,\ldots, $b_{g(j)}(j)$ denote the death
times, after which there are remain only $j$ individuals in the
population. Apparently, based on the convention $g(0)=1$, the times
$a_1(0)$ and $b_1(0)$ are unique, $a_1(0)$ is the moment of the
first birth and $b_1(0)$ is the extinction time. For $1\leq j$, the
time intervals
\begin{equation}\label{eq-5.1}
\big[a_1(j), b_1(j)\big), \big[a_2(j), b_2(j)\big), \ldots,
\big[a_{g(j)}(j), b_{g(j)}(j)\big)
\end{equation}
are contained in the intervals
\begin{equation}\label{eq-5.2}
\begin{aligned}
&\big[a_1(j-1), b_1(j-1)\big), \big[a_2(j-1), b_2(j-1)\big),
\ldots,\\
&\big[a_{g(j-1)}(j-1), b_{g(j-1)}(j-1)\big).
\end{aligned}
\end{equation}
Let us delete the intervals of \eqref{eq-5.1} from those of
\eqref{eq-5.2} and merge the ends. Then, according to the property
of the lack of memory of exponential distribution, the number of
merged points in each of the intervals of \eqref{eq-5.2} coincides
in distribution with the number of births per a death time of an
individual in the population (the birth and death rates are reckoned
to be unchanged and equal to $\lambda_{j}$ and $\mu_{j}$,
respectively) and has geometric distribution with parameter
$\lambda_{j}/(\lambda_{j}+\mu_{j})$, and given that the number of intervals in $m$, the total number of merged points in these intervals has negative binomial distribution, that is,
$$
\mathsf{P}\{g(j+1)=k|g(j)=m\}=\binom{m+k-1}{k}\left(\frac{\lambda_j}{\lambda_j+\mu_j}\right)^k\left(\frac{\mu_j}{\lambda_j+\mu_j}\right)^m.
$$
This enables us to conclude
that $\{g(j)\}$ has a structure similar to that of the branching process. Specifically,
the distribution of the number of offspring in the $j$th generation
depends on the order number of generation, $j$, and is described by
the Markov chain $V_j$.

Hence, the distribution of $g(j)$ coincides with the
distribution of $V_j$, and \eqref{eq-5.3}, \eqref{eq-5.4} and \eqref{eq-5.5} are true.
\end{proof}

\subsection{Proof of \eqref{eq-2.11}, \eqref{eq-2.12} and \eqref{eq-2.10}} Applying Lemma \ref{l1} with $\lambda_n=p_n(2)$ and $\mu_n=q_n(2)$ we arrive at \eqref{eq-2.11} with the Markov chain $R_n$ defined in Section \ref{S3.3}.
To prove \eqref{eq-2.12}, note that $\overleftarrow{f}[\mathcal{N}(n)]=\overrightarrow{f}[\mathcal{N}(n+1)]$,
since the total number of crossings from the level $n$ to $n+1$ must
coincide with the total number of crossings from the level $n+1$ to
$n$. Next,
$f[\mathcal{N}(n)]=\overrightarrow{f}[\mathcal{N}(n)]+\overleftarrow{f}[\mathcal{N}(n)]$
and \eqref{eq-2.10} is true.

\section{Proof of Theorem \ref{t1}}\label{S6}

\subsection{Prelude}
The proof is based on study of the random walk in $\mathbb{Z}^2_+$, which is defined as follows:
\begin{eqnarray}
\mathbf{S}_0&=&\mathbf{0},\label{eq-6.4}\\
\mathbf{S}_{t}&=&\mathbf{S}_{t-1}+\mathbf{e}_t(\mathbb{Z}_+^2), \quad
t\geq1,\label{eq-6.5}
\end{eqnarray}
where
\begin{equation}\label{eq-6.6}
{\mathbf{e}}_t(\mathbb{Z}_+^2)=\begin{cases} \mathbf{e}_t, &\text{if}\
 {S}_{t-1}^{(i)}+e_t^{(i)}\geq0 \ \text{for all} \ i=1,2;\\
-\mathbf{e}_t, &\text{if}\
{S}_{t-1}^{(i)}+e^{(i)}_t=-1 \ \text{for a certain} \ i=1,2,
\end{cases}
\end{equation}
and the vector $\mathbf{e}_t=\Big(e_t^{(1)},
e_t^{(2)}\Big)=\mathbf{e}_t(\mathbb{Z}^2)$.

The random walk defined by \eqref{eq-6.4} -- \eqref{eq-6.6} is the reflected version of the random walk defined by \eqref{eq-1.1} and \eqref{eq-1.2}.

Hence, the  distributions of
$\overrightarrow{f}(\mathbf{n})$, $\overleftarrow{f}(\mathbf{n})$
and $f(\mathbf{n})$ for the random walk in $\mathbb{Z}^2_+$  are
equivalent to the  distributions of
$\overrightarrow{f}[\mathcal{X}(\mathbf{n})]$,
$\overleftarrow{f}[\mathcal{X}(\mathbf{n})]$ and
$f[\mathcal{X}(\mathbf{n})]$, respectively, for the random walk in
$\mathbb{Z}^2$.

According to the conventional notation, $\tau$ is the time of the first return to the origin of the random walk in $\mathbb{Z}^2$. Apparently, if $\tau^\prime$ is the time of the first return to the origin of the random walk in $\mathbb{Z}_+^2$, the random times $\tau$ and $\tau^\prime$ coincide in distribution, and being considered on the same probability space can be considered as identical. Hence, the only notation $\tau$ will be used for either of random walks.

\subsection{Proof of \eqref{eq-3.1}}\label{S6.2}
The random walk in $\mathbb{Z}_+^2$ is modeled by two independent
queueing systems, the structure of which is described in Section
\ref{S5.2}.

At the initial time moment $t=0$ both queueing systems are assumed to
be empty, and, after $t=0$, the first arrival to one of the queueing
systems occurs. By busy period we mean the time interval $[1,
\tau]$. (Recall that the time $t$ is discrete, $t=1$ is the moment
of the first arrival to one of the two systems, which are
initially empty, and $\tau$ means the $\tau$'s event of the Poisson
process, at which the queueing systems become empty once again
at the first time since $t=0$.)

The proof of \eqref{eq-3.1} is based on an extension of the level-crossing
technique given in the proof of Lemma \ref{l3} for the
two-dimensional random walk.

Consider first the set of vectors $\mathcal{N}^+(n)$ (the set of all
vectors in $\mathbb{Z}^2_+$ with norm $n$). The arguments for this
set of vectors is similar to that provided before in the proof of
Lemma \ref{l3}. Recall them in terms of the relevant objects and
notation. Let $z_n$ denote the number of cases when at the
moment of arrival of a customer, the total (cumulative) number of
customers in the queueing systems becomes $n$. Let $a_1(n)$,
$a_2(n)$,\ldots, $a_{z_n}(n)$ be the moments of
these arrivals, and let $b_1(n)$, $b_2(n)$,\ldots,
$b_{z_n}(n)$ be the moments of service completion, when
there totally remain $n-1$ customers.

Apparently, the time intervals
\begin{equation}\label{eq-6.30}
\big[a_1(n), b_1(n)\big), \big[a_2(n),
b_2(n)\big),\ldots, \big[a_{z_n}(n),
b_{z_n}(n)\big)
\end{equation}
are contained in the time intervals
\begin{equation}\label{eq-6.31}
\begin{aligned}
&\big[a_1(n-1), b_1(n-1)\big), \big[a_2(n-1),
b_2(n-1)\big)
,\ldots,\\
&\big[a_{z_{n-1}}(n-1),
b_{z_{n-1}}(n-1)\big).
\end{aligned}
\end{equation}
Deleting the intervals of \eqref{eq-6.30} from those of
\eqref{eq-6.31} and merging the ends yields the set of points. The
residual times in the points intervals merged have an exponential
distribution, the parameter of which typically depends on the
allocation structure of $n-1$ customers in two queueing systems at
the moment of the service completion. For instance, if one of the servers
is empty, then the (residual) service rate is $\lambda$ and
total (residual) arrival rate is $3\lambda$. (The last
includes the rate $\lambda$ of a negative service, so $2\lambda+\lambda=3\lambda$.)

Let $\mathbf{n}$ be a point that has norm $\|\mathbf{n}\|=n$
 and
thus belongs to $\mathcal{N}^+(n)$.
 Denote by $z_{\mathbf{n}}$ the number of
cases when at the moment of a customer's arrival to one of the
queueing systems, there become $n^{(1)}$, $n^{(2)}$ customers in the corresponding queueing systems, the order
numbers of which is indicated by the upper index, and $n^{(1)}+n^{(2)}=n$. Let
$a_1(\mathbf{n})$, $a_2(\mathbf{n})$,\ldots,
$a_{z_{\mathbf{n}}}(\mathbf{n})$ be the moments of these
arrivals, and let $b_1(\mathbf{n})$,
$b_2(\mathbf{n})$,\ldots,
$b_{z_{\mathbf{n}}}(\mathbf{n})$ be the moments of
service completions following the first time after the corresponding
times $a_i(\mathbf{n})$, $i=1,2,\ldots,z_{\mathbf{n}}$,
when there remain $n-1$ customers in two queueing
systems in total.

So, we have the time intervals
\begin{equation}\label{eq-6.7}
\big[a_1(\mathbf{n}), b_1(\mathbf{n})\big),
\big[a_2(\mathbf{n}), b_2(\mathbf{n})\big),\ldots,
\big[a_{z_{\mathbf{n}}}(\mathbf{n}),
b_{z_{\mathbf{n}}}(\mathbf{n})\big).
\end{equation}
Note, that if $\mathbf{n}=\mathbf{1}_i$ (with $n=1$) the time interval $
\big[a_1(\mathbf{1}_i), b_1(\mathbf{1}_i)\big)$, if
it is, coincides with the busy period $[1,\tau]$, and
\begin{equation}\label{eq-6.11}
\mathsf{P}\left\{z_{\mathbf{1}_i}=1\right\}=\frac{1}{2}, \quad
\mathsf{P}\left\{z_{\mathbf{1}_i}=0\right\}=\frac{1}{2},
\quad i=1,2,
\end{equation}
together with the condition
\begin{equation}\label{eq-6.15}
z_{\mathbf{1}_1}+z_{\mathbf{1}_2}=1.
\end{equation}
Note, that any two differences
$a_2(\mathbf{n})-a_1(\mathbf{n})$ and
$a_{i+1}(\mathbf{n})-a_i(\mathbf{n})$ ($i\geq2$) (if
exist) are identically distributed, so $a_i(\mathbf{n})$ have
a structure of regeneration points.

Now, let $\mathbf{m}\in\mathcal{M}^-(\mathbf{n})$, and let
$z_{\mathbf{m}}$ denote the number of arrivals, at the time of
which there become $m^{(1)}$, $m^{(2)}$ numbers of
customers in the corresponding queueing systems, the order numbers
of which are indicated by the upper index, $m^{(1)}+m^{(2)}=n-1$. Then, for each of the
vectors $\mathbf{m}$ from the set $\mathcal{M}^-(\mathbf{n})$ one
can define the sequences $a_1(\mathbf{m})$,
$a_2(\mathbf{m})$,\ldots,
$a_{z_{\mathbf{m}}}(\mathbf{m})$ and
$b_1(\mathbf{m})$, $b_2(\mathbf{m})$,\ldots,
$b_{z_{\mathbf{m}}}(\mathbf{m})$ and the intervals
\begin{equation}\label{eq-6.24}
\big[a_1(\mathbf{m}), b_1(\mathbf{m})\big),
\big[a_2(\mathbf{m}), b_2(\mathbf{m})\big),\ldots,
\big[a_{z_{\mathbf{m}}}(\mathbf{m}),
b_{z_{\mathbf{m}}}(\mathbf{m})\big)
\end{equation}
by the similar way as before.

Apparently, there are intervals
defined by \eqref{eq-6.7} that are contained in the set of intervals
\eqref{eq-6.24}, and let their number be
$z_{\mathbf{m},\mathbf{n}}$.
Let us delete the intervals of \eqref{eq-6.7} from those of
\eqref{eq-6.24} and merge the ends. Note, that with
$n=\|\mathbf{n}\|$ the set of intervals defined by \eqref{eq-6.24}
is a subset of the system of intervals given by \eqref{eq-6.31}, and
the set of intervals given by \eqref{eq-6.7} is a subset of
intervals given by \eqref{eq-6.30}. Let us remove all intervals of
\eqref{eq-6.31} that are not \eqref{eq-6.24} and all intervals of
\eqref{eq-6.30} that are not \eqref{eq-6.7}. Then, the
aforementioned merged points of the intervals of \eqref{eq-6.7}
imbedded into the intervals of \eqref{eq-6.24} have a structure of
regeneration points (the differences between merged points have the
same distribution) and satisfy the following properties. First, the
residual times to the next arrival or service completion both
distributed exponentially.

The mean time to the next arrival that
occurs from state $\mathbf{m}\in\mathcal{M}^-(\mathbf{n})$ to the
state $\mathbf{n}$ depends on the state $\mathbf{m}$. More
specifically, if $\mathbf{m}$ is on boundary and state $\mathbf{n}$ has rank 2, i.e. $r(\mathbf{n})=2$, then, because of reflection at zero, the mean time to the next arrival is $1/(2\lambda)$. In all other situations it is $1/\lambda$. (See Figure 3.)

 Hence, the number of merged
points within an arbitrary interval $\big[a_j(\mathbf{m}),
b_j(\mathbf{m})\big)$, due to the property of the lack of
memory of exponential distribution, has a geometric distribution,
which is the same for any $j$. The parameter of this geometric
distribution depends on $\mathbf{m}$. If $\mathbf{m}\in\mathcal{M}^-(\mathbf{n})$, $d_0(\mathbf{m})=1$ (that is, the vector $\mathbf{m}$ belongs to the boundary) and $\|\mathbf{n}\|\geq2$, then the parameter of geometric distribution is $r(\mathbf{n})$. Otherwise, if $d_0(\mathbf{m})=0$ (the vector $\mathbf{m}$ is in interior) and $\|\mathbf{n}\|\geq2$, it is 1/4.

The next step is to distinguish different elements of $\mathcal{M}^-(\mathbf{n})$. The number of vectors in $\mathcal{M}^-(\mathbf{n})$ can be one or two. If $\mathbf{n}$ is on boundary, than $\mathcal{M}^-(\mathbf{n})$ contains only a single vector. Otherwise, there two vectors in $\mathcal{M}^-(\mathbf{n})$. To be formal, let us denote the elements of $\mathcal{M}^-(\mathbf{n})$ by $\mathbf{m}_1$,\ldots, $\mathbf{m}_{2-d_0(\mathbf{n})}$. Then,
$$
z_{\mathbf{n}}=z_{\mathbf{m}_1, \mathbf{n}}+\ldots+z_{\mathbf{m}_{2-d_0(\mathbf{n})}, \mathbf{n}},
$$
where $z_{\mathbf{m}_1, \mathbf{n}}$,\ldots,
$z_{\mathbf{m}_{2-d_0(\mathbf{n})}, \mathbf{n}}$ are
independent random variables. The construction of the proof implies the required properties of $q(\mathbf{m},\mathbf{n})$. Indeed, if $d_0(\mathbf{m})=1$ and $\|\mathbf{n}\|\geq2$, then
$$
\mathsf{P}\{z_{\mathbf{m},\mathbf{n}}=n~|~z_{\mathbf{m}}=k\}=\binom{k+n-1}{n}\left(\frac{r(\mathbf{n})}{4}\right)^n\left(\frac{4-r(\mathbf{n})}{4}\right)^k.
$$
Otherwise, if $d_0(\mathbf{m})=0$ and $\|\mathbf{n}\|\geq2$, then
$$
\mathsf{P}\{z_{\mathbf{m},\mathbf{n}}=n~|~z_{\mathbf{m}}=k\}=\binom{k+n-1}{n}\left(\frac{1}{4}\right)^n\left(\frac{3}{4}\right)^k.
$$

\subsection{Proof of \eqref{eq-3.2} and \eqref{eq-3.3}}
The proof of \eqref{eq-3.2} is technically similar to that of
\eqref{eq-3.1}. We define the system of intervals \eqref{eq-6.24}
where $\mathbf{m}$ belongs now to $|\mathcal{M}^+(\mathbf{n})|$.
Then, the intervals of \eqref{eq-6.7} are contained in those of
\eqref{eq-6.24}. Deleting \eqref{eq-6.7} from \eqref{eq-6.24} and
merging the ends, we obtain the set of points, the number of which
is denoted by $z_{\mathbf{m},\mathbf{n}}$. Similarly, we define $z_{\mathbf{n}}$ satisfying the relation
$$
z_{\mathbf{n}}=z_{\mathbf{m}_1, \mathbf{n}}+z_{\mathbf{m}_2, \mathbf{n}},
$$
where $\mathbf{m}_1$ and $\mathbf{m}_2$ are the vectors of $|\mathcal{M}^+(\mathbf{n})|$. (The total number of vectors in $|\mathcal{M}^+(\mathbf{n})|$ is two.)

By similar
arguments to those used in the proof in Section \ref{S6.2}, we
arrive at the relation
$$
\mathsf{P}\{z_{\mathbf{m},\mathbf{n}}=n~|~z_{\mathbf{m}}=k\}=\binom{k+n-1}{n}\left(\frac{1}{4}\right)^n\left(\frac{3}{4}\right)^k.
$$
So, \eqref{eq-3.2} follows.

 In turn, the proof of
\eqref{eq-3.3} follows from the relation
$$
f[\mathcal{X}(\mathbf{n})]=\overrightarrow{f}[\mathcal{X}(\mathbf{n})]
+\overleftarrow{f}[\mathcal{X}(\mathbf{n})].
$$
The proof of Theorem \ref{t1} is completed.

\section{Proof of Theorem \ref{t3}}\label{S7}

The proof of Theorem \ref{t3} is similar to that of Theorem
\ref{t1}. In this case, however, the random walk is modelled by the two independent and identical $M/M/1$ queueing systems, in which the arrival and service rates are equal, and \textit{no negative service is considered}. This is because \textit{no reflection mechanism is presented}, and the two queueing systems represent the only part of the basic random walk in the main quarter plane.
Then, the behavior in boundary states is reckoned to be similar to those in interior states. For instance, if the random walk moves from the state $\mathbf{m}=(0, 5)$ to the state
$\mathbf{n}=(-1, 5)$, it still can be modelled by the two queueing systems, in the first of which the component $(-1)$ of the random walk is transformed to $+1$ in the queue, but the movement from ``empty queue" to ``queue with a customer" occurs without ``reflection mechanism". Specifically, because of the symmetry, we use the fact that the number of crossings a state containing a negative component or negative components has the same distribution as the number of crossings the corresponding state from the main quarter plane. Hence, the proof of Theorem \ref{t1} can be adapted to this case in very similar way. Specifically, the arguments of the proof of \eqref{eq-3.4} are fully similar to those of \eqref{eq-3.1} in Theorem \ref{t1}. However, the proof of \eqref{eq-3.5} as well as \eqref{eq-3.6} need to take into account the boundary effect that is explained below.

The difference between the proofs of \eqref{eq-3.2} and \eqref{eq-3.5} is that, instead of the vectors of the set $|\mathcal{M}^+(\mathbf{n})|$ used in the proof of \eqref{eq-3.2} in Theorem \ref{t1}, the proof of \eqref{eq-3.5} should use the vectors of the set $\mathcal{M}^+(\mathbf{n})$. If $\mathbf{n}$ is from interior, then the sets $\mathcal{M}^+(\mathbf{n})$ and $|\mathcal{M}^+(\mathbf{n})|$ coincide, and the proof of \eqref{eq-3.5} in this case is the same as that of \eqref{eq-3.2}.
If, however, $\mathbf{n}$ belongs to the boundary, then $\mathcal{M}^+(\mathbf{n})$ contains a vector with negative component. For instance, if $\mathbf{n}=(0, 5)$, then $\mathcal{M}^+(\mathbf{n})=\{(1, 5), (0,6), (-1, 5)\}$. The value $(-1)$ in the vector $(-1, 5)$ is not associated with negative queue. The vector $(-1, 5)$ is considered as a tagged state that is formally added to the set of states. Here is a more detailed explanation.

Let $z_{\mathbf{n}}$ denote the number of
cases when at the moment of a customer's service completion in one of the
queueing systems, there become $n^{(1)}$, $n^{(2)}$ customers in the corresponding queueing systems, the order
numbers of which is indicated by the upper index, and $n^{(1)}+n^{(2)}=n$. Considering the interesting us case, set $n^{(1)}=0$ and $n^{(2)}=n$.

With $\mathbf{n}=(0, n)$ let
$a_1(\mathbf{n})$, $a_2(\mathbf{n})$,\ldots,
$a_{z_{\mathbf{n}}}(\mathbf{n})$ be the moments of these
service completions, and let $b_1(\mathbf{n})$,
$b_2(\mathbf{n})$,\ldots,
$b_{z_{\mathbf{n}}}(\mathbf{n})$ be the moments of
arrivals following the first time after the corresponding
times $a_i(\mathbf{n})$, $i=1,2,\ldots,z_{\mathbf{n}}$,
when there are $n+1$ customers in two queueing
systems in total.

So, we have the time intervals
\begin{equation}\label{eq-7.1}
\big[a_1(\mathbf{n}), b_1(\mathbf{n})\big),
\big[a_2(\mathbf{n}), b_2(\mathbf{n})\big),\ldots,
\big[a_{z_{\mathbf{n}}}(\mathbf{n}),
b_{z_{\mathbf{n}}}(\mathbf{n})\big).
\end{equation}

Now, let $\mathbf{m}=\big(m^{(1)}, m^{(2)}\big)=(1, n)$, and let
$z_{\mathbf{m}}$ denote the number of service completions, at the time of
which there become $m^{(1)}$, $m^{(2)}$ numbers of
customers in the corresponding queueing systems, the order numbers
of which are indicated by the upper index.

Then, one
can define the sequences $a_1(\mathbf{m})$,
$a_2(\mathbf{m})$,\ldots,
$a_{z_{\mathbf{m}}}(\mathbf{m})$ and
$b_1(\mathbf{m})$, $b_2(\mathbf{m})$,\ldots,
$b_{z_{\mathbf{m}}}(\mathbf{m})$ and the intervals
\begin{equation}\label{eq-7.2}
\big[a_1(\mathbf{m}), b_1(\mathbf{m})\big),
\big[a_2(\mathbf{m}), b_2(\mathbf{m})\big),\ldots,
\big[a_{z_{\mathbf{m}}}(\mathbf{m}),
b_{z_{\mathbf{m}}}(\mathbf{m})\big)
\end{equation}
by the similar way as before. Now let $z^\prime_{\mathbf{m}^\prime}$ be a random variable having the same distribution as $z_{\mathbf{m}}$, and let
\begin{equation}\label{eq-7.3}
\big[a_1(\mathbf{m}^\prime), b_1(\mathbf{m}^\prime)\big),
\big[a_2(\mathbf{m}^\prime), b_2(\mathbf{m}^\prime)\big),\ldots,
\big[a_{z^\prime_{\mathbf{m}^\prime}}(\mathbf{m}^\prime),
b_{z^\prime_{\mathbf{m}^\prime}}(\mathbf{m}^\prime)\big)
\end{equation}
be the tagged intervals, the total number of which is $z^\prime_{\mathbf{m}^\prime}$.

The number of intervals \eqref{eq-7.1} that are contained in \eqref{eq-7.2} are denoted by $z_{\mathbf{m},\mathbf{n}}$, and they are associated with the transition $(1, n)\to(0,n)$. Assuming that the number of  the virtual transition $(-1,n)\to(0,n)$ has the same distribution as $z_{\mathbf{m},\mathbf{n}}$ and independent of it, denote it by
$z^\prime_{\mathbf{m}^\prime,\mathbf{n}}$. Then
\begin{equation}\label{eq-7.6}
z_{\mathbf{n}}=\sum_{\mathbf{k}\in|\mathcal{M}^+(\mathbf{n})|}z_{\mathbf{k},\mathbf{n}},
\end{equation}
and let
\begin{equation}\label{eq-7.4}
z^*_{\mathbf{n}}=z_{\mathbf{n}}+z^\prime_{\mathbf{m}^\prime,\mathbf{n}}.
\end{equation}
In \eqref{eq-7.6} and \eqref{eq-7.4} we first find the sum over $|\mathcal{M}^+(\mathbf{n})|$ and then add the element $z^\prime_{\mathbf{m}^\prime,\mathbf{n}}$ for the reason that in the model with two queueing system we cannot provide summation over $\mathcal{M}^+(\mathbf{n})$ in the boundary case directly.

If $\mathbf{n}$ is from the interior, then
\begin{equation}\label{eq-7.5}
z_{\mathbf{n}}=\sum_{\mathbf{k}\in\mathcal{M}^+(\mathbf{n})}z_{\mathbf{k},\mathbf{n}},
\end{equation}
and the distributions of $z^*_{\mathbf{n}}$ given by \eqref{eq-7.4} and $z_{\mathbf{n}}$ given by \eqref{eq-7.5} are the same.

\section{Proof of Theorem \ref{t4}}\label{S8}
Note, first that in the case of the one-dimensional random walk the results automatically follow
from Proposition \ref{c2}. Indeed, with $\mathsf{E}Z_1=1$
the branching process $Z_n$ is a martingale, and
all the required results follow from this property. In the case
of the two-dimensional random walk, the results cannot be retrieved from Theorems \ref{t1}
or \ref{t3} by the similar way. The properties of random fields
reducing them directly to martingales similarly to those of random
processes are unknown. Hence, the proof of this theorem should be
produced independently of the statements of Theorems \ref{t1} and \ref{t3}.

Consider the model of two independent queueing systems of Section \ref{S5.2}. It follows from the results of the Appendix, specifically from (A.6) that if $\mathsf{E}\{f[\mathcal{X}(\mathbf{n}_0)]\}<\infty$ for some vector $\mathbf{n}_0\in\mathbb{Z}^2_+\setminus\{{\mathbf{0}}\}$, then for any two vectors $\mathbf{n}_1\in\mathbb{Z}^2_+\setminus\{\mathbf{0}\}$ and $\mathbf{n}_2\in\mathbb{Z}^2_+\setminus\{\mathbf{0}\}$
\begin{equation}\label{eq-8.5}
\mathsf{E}\{f[\mathcal{X}(\mathbf{n}_1)]\}=
2^{d_0(\mathbf{n}_2)-d_0(\mathbf{n}_1)}\mathsf{E}\{
f[\mathcal{X}(\mathbf{n}_2)]\}.
\end{equation}
Taking for instance $\mathbf{n}_0=\mathbf{1}_1$ it is readily seen from Properties (2) and (3) in Section \ref{S3.2} that $\mathsf{E}\{f[\mathcal{X}(\mathbf{1}_1)]\}<\infty$, and, hence, \eqref{eq-8.5} is true.

It follows from \eqref{eq-8.5} that for all $\mathbf{n}\geq\mathbf{1}$,
\begin{equation}\label{eq-8.1}
\mathsf{E}\{f[\mathcal{X}(\mathbf{n})]\}=\mathsf{E}\{f[\mathcal{X}(\mathbf{1})]\}.
\end{equation}
From same Relation \eqref{eq-8.5} we have
\begin{equation}\label{eq-8.2}
\mathsf{E}\{f[\mathcal{X}(\mathbf{1}_1)]\}=\mathsf{E}\{f[\mathcal{X}(\mathbf{1}_2)]\}=\frac{1}{2}\mathsf{E}\{f[\mathcal{X}(\mathbf{1})]\}.
\end{equation}
In turn, from \eqref{eq-8.1} and \eqref{eq-8.2} we can establish the similar relationships for $\mathsf{E}\{f(\mathbf{n})\}$, $\mathbf{n}\in\mathbb{Z}^2\setminus\{\mathbf{0}\}$. Taking into account that
$$
\mathsf{E}\{f[\mathcal{X}(\mathbf{n})]\}=\sum_{\mathbf{m}\in\mathcal{X}(\mathbf{n})}\mathsf{E}\{f(\mathbf{m})\}.
$$
we have as follows. For all $\mathbf{n}\geq\mathbf{1}$ from \eqref{eq-8.1} we obtain
\begin{equation}\label{eq-8.3}
\mathsf{E}\{f(\mathbf{n})\}=\frac{1}{4}\mathsf{E}\{f[\mathcal{X}(\mathbf{1})]\}.
\end{equation}
In turn, from \eqref{eq-8.2} we obtain
\begin{equation}\label{eq-8.4}
\mathsf{E}\{f(\mathbf{1}_1)\}=\mathsf{E}\{f(\mathbf{1}_2)\}=\frac{1}{2}\times\frac{1}{2}\mathsf{E}\{f[\mathcal{X}(\mathbf{1})]\}=\frac{1}{4}\mathsf{E}\{f[\mathcal{X}(\mathbf{1})]\}.
\end{equation}
So, combining \eqref{eq-8.3} and \eqref{eq-8.4}, we arrive at the conclusion that for all $\mathbf{n}\in\mathbb{Z}^2\setminus\{\mathbf{0}\}$
\begin{equation}\label{eq-8.11}
\mathsf{E}\{f(\mathbf{n})\}=\frac{1}{4}\mathsf{E}\{f[\mathcal{X}(\mathbf{1})]\}=c.
\end{equation}
Our aim now is to prove that the constant $c$ in \eqref{eq-8.11} is equal to one.

Prove first \eqref{eq-3.7} and \eqref{eq-3.8}. We have
$$
\mathsf{E}\left\{\overrightarrow{f}(\mathbf{1}_1)\right\}=\mathsf{E}\left\{\overrightarrow{f}(-\mathbf{1}_1)\right\}=\mathsf{E}\left\{\overrightarrow{f}(\mathbf{1}_2)\right\}
=\mathsf{E}\left\{\overrightarrow{f}(-\mathbf{1}_2)\right\}=\frac{1}{4}.
$$
Let $\mathbf{n}\in\mathbb{Z}^2\setminus\{\mathbf{0}\}$. If $\mathbf{n}$ is on the boundary, then the set $\mathcal{M}^-(\mathbf{n})$ contains only a single vector, while the set $\mathcal{M}^+(\mathbf{n})$ contains three vectors. If $\mathbf{n}$ is from the interior, then each of the sets $\mathcal{M}^-(\mathbf{n})$ and $\mathcal{M}^+(\mathbf{n})$ contains two vectors (see Figure 1). Since all states are equally likely, we obtain:
\begin{equation*}
\mathsf{E}\left\{\overrightarrow{f}(\mathbf{n})\right\}=\begin{cases}\mathsf{E}\big\{\overrightarrow{f}(\mathbf{1}_1)\big\}=\frac{1}{4}, &\text{if} \ \mathbf{n} \ \text{on boundary},\\
2\mathsf{E}\big\{\overrightarrow{f}(\mathbf{1}_1)\big\}=\frac{1}{2}, &\text{if} \ \mathbf{n} \ \text{in interior},
\end{cases}
\end{equation*}
\begin{equation*}
\mathsf{E}\left\{\overleftarrow{f}(\mathbf{n})\right\}=\begin{cases}3\mathsf{E}\big\{\overrightarrow{f}(\mathbf{1}_1)\big\}=\frac{3}{4}, &\text{if} \ \mathbf{n} \ \text{on boundary},\\
2\mathsf{E}\big\{\overrightarrow{f}(\mathbf{1}_1)\big\}=\frac{1}{2}, &\text{if} \ \mathbf{n} \ \text{in interior}.
\end{cases}
\end{equation*}
Since $\mathsf{E}\{f(\mathbf{n})\}=\mathsf{E}\big\{\overrightarrow{f}(\mathbf{n})\big\}+\mathsf{E}\big\{\overleftarrow{f}(\mathbf{n})\big\}$, then the constant $c$ in \eqref{eq-8.11} must be equal to 1.

\section{Crossings states and sets of states in random walks defined in finite areas}\label{S9}
The aim of this section is to develop the results on crossings states and sets of states for random walks of higher dimension than two. In this section we consider random walks in $[-N,N]^d$ and $[0,N]^d$. The random walk in $[-N,N]^d$ is defined as follows:
\begin{eqnarray}
\mathbf{S}_0&=&\mathbf{0},\label{eq-6.1}\\
\mathbf{S}_{t}&=&\mathbf{S}_{t-1}+\mathbf{e}_t([-N,N]^d), \quad
t\geq1,\label{eq-6.2}
\end{eqnarray}
where in \eqref{eq-6.1} and later on $\mathbf{0}$ is the $d$-dimensional zero and the vector $\mathbf{e}_t([-N,N]^d)$ is
\begin{equation*}
\mathbf{e}_t([-N,N]^d)=\begin{cases}
\mathbf{e}_t, &\mbox{if}\  |\mathbf{S}_{t-1}+\mathbf{e}_t|\leq N\mathbf{1};\\
\mathbf{0}, &\mbox{otherwise},
\end{cases}
\end{equation*}
and $\mathbf{e}_t=\mathbf{e}_t(\mathbb{Z}^d)$, where the random vector $\mathbf{e}_t(\mathbb{Z}^d)$ for the random walk in $\mathbb{Z}^d$ is defined similarly to that $\mathbf{e}_t(\mathbb{Z}^2)$ is defined for the random walk in $\mathbb{Z}^2$. Namely, let $\mathbf{1}_{i}$ denote the vector, the
$i$th component of which is 1, and the rest component are 0. Then,
the vector $\mathbf{e}_t(\mathbb{Z}^d)$ is one of the $2d$ vectors
$\{\pm\mathbf{1}_{i}, i=1,2,\ldots,d\}$ that is randomly chosen with
probability $1/(2d)$ independently of the other vectors and the
history of the random walk.

The random walk in $[0,N]^d$ is, in turn, the reflected version of the random walk in $[-N,N]^d$, and it is defined as
\begin{eqnarray*}
\mathbf{S}_0&=&\mathbf{0},\\
\mathbf{S}_{t}&=&\mathbf{S}_{t-1}+\mathbf{e}_t([0,N]^d), \quad
t\geq1,
\end{eqnarray*}
where the vector $\mathbf{e}_t([0,N]^d)$ is
\begin{equation*}
{\mathbf{e}}_t([0,N]^d)=\begin{cases} \mathbf{e}_t, &\text{if}\
{\mathbf{S}}_{t-1}+\mathbf{e}_t\leq N\mathbf{1},\\ &\text{and}\
 {S}_{t-1}^{(i)}+e_t^{(i)}\geq0 \ \text{for all} \ i=1,2,\ldots,d;\\
-\mathbf{e}_t, &\text{if}\
{\mathbf{S}}_{t-1}+\mathbf{e}_t\leq N\mathbf{1},\\ &\text{and}\
{S}_{t-1}^{(i)}+e^{(i)}_t=-1\ \text{for one of} \ i=1,2,\ldots,d;\\
\mathbf{0}, &\mbox{otherwise},
\end{cases}
\end{equation*}
and $\mathbf{e}_t=\big(e_t^{(1)},
e_t^{(2)},\ldots, e_t^{(d)}\big)=\mathbf{e}_t(\mathbb{Z}^d)$.

Note first, that all the definitions of Section \ref{S2} can be automatically developed to the $d$-dimensional random walks considered here. The definition of $\mathcal{X}(\mathbf{n})$, $\mathcal{N}^+(n)$, $\mathcal{N}(n)$, lower and upper sets $\mathcal{M}^-(\mathbf{n})$ and $\mathcal{M}^+(\mathbf{n})$, positive upper set $|\mathcal{M}^+(\mathbf{n})|$, crossings the sets $f(\mathcal{Z})$, $\overrightarrow{f}(\mathcal{Z})$ and $\overleftarrow{f}(\mathcal{Z})$ and all other derivative notions remain similar to those in Section \ref{S2} given originally for the two-dimensional random walk. For instance, the total number of elements in the sets $\mathcal{M}^-(\mathbf{n})$ and $\mathcal{M}^+(\mathbf{n})$ are $d-d_0(\mathbf{n})$ and $d+d_0(\mathbf{n})$, respectively, and the total number of elements of the set $|\mathcal{M}^+(\mathbf{n})|$ is $d$.

The concepts of the Markov fields $P_{\mathbf{n}}$, $Q_{\mathbf{n}}$ and the Markov chain $R_n$ can be also developed. Specifically, the properties of the Markov fields $P_{\mathbf{n}}$, $Q_{\mathbf{n}}$ the Markov chain $R_n$ are given below.

For the field $P_{\mathbf{n}}$ and $p(\mathbf{m},\mathbf{n})$ we have the properties:
\begin{enumerate}
  \item $P_{\mathbf{0}}=1$.
  \smallskip
  \item $\mathsf{P}\{p(\mathbf{0},\mathbf{1}_i)=1\}=\mathsf{P}\{p(\mathbf{0},-\mathbf{1}_i)=1\}=
1/(2d)$, $i=1,2,\ldots,d$,\\ and
$\sum_{i=1}^{d}\big(p(\mathbf{0},\mathbf{1}_i)+p(\mathbf{0},-\mathbf{1}_i)\big)=1.$

  \smallskip
  \item $\mathsf{P}\{p(\mathbf{1}_i, \mathbf{0})=1\}=\mathsf{P}\{p(-\mathbf{1}_i, \mathbf{0})=1\}=
1/(2d)$, $i=1,2,\ldots,d$,\\ and
  $\sum_{i=1}^{d}\big(p(\mathbf{0}\mathbf{1}_i, \mathbf{0})+p(-\mathbf{1}_i, \mathbf{0})\big)=1.$

  \smallskip

  \item if $\mathbf{m}\in\mathcal{M}^-(\mathbf{n})$, $\|\mathbf{n}\|\geq2$, then
  $$
  \mathsf{P}\left\{p(\mathbf{m},\mathbf{n})=n~\Big|~\sum_{\mathbf{m}^*\in\mathcal{M}^-(\mathbf{m})}p(\mathbf{m}^*,\mathbf{m})=k\right\}=\binom{k+n-1}{n}\left(\frac{1}{2d}\right)^n\left(\frac{2d-1}{2d}\right)^k.
  $$
  \smallskip
  \item if $\mathbf{m}\in\mathcal{M}^+(\mathbf{n})$, $\mathbf{n}\neq\mathbf{0}$, then
  $$
  \mathsf{P}\left\{p(\mathbf{m},\mathbf{n})=n~\Big|~\sum_{\mathbf{m}^*\in\mathcal{M}^+(\mathbf{m})}p(\mathbf{m}^*,\mathbf{m})=k\right\}=\binom{k+n-1}{n}\left(\frac{1}{2d}\right)^n\left(\frac{2d-1}{2d}\right)^k.
  $$

\end{enumerate}

For the field $Q_{\mathbf{n}}$ and $q(\mathbf{m},\mathbf{n})$ we have the properties:
\begin{enumerate}
  \item $Q_{\mathbf{0}}=1$.
  \smallskip
  \item $\mathsf{P}\{q(\mathbf{0},\mathbf{1}_i)=1\}=
1/d$, $i=1,2,\ldots, d$, and $\sum_{i=1}^{d}q(\mathbf{0},\mathbf{1}_i)=1.$

  \smallskip
  \item $\mathsf{P}\{q(\mathbf{1}_i,\mathbf{0})=1\}
=1/d$, and $\sum_{i=1}^{d}q(\mathbf{1}_i,\mathbf{0})=1.$

  \smallskip
  \item if $\mathbf{m}\in|\mathcal{M}^+(\mathbf{n})|$ and $\mathbf{n}\neq\mathbf{0}$, then
  $$
  \begin{aligned}
  &\mathsf{P}\left\{q(\mathbf{m},\mathbf{n})=n~\Big|~\sum_{\mathbf{m}^*\in|\mathcal{M}^+(\mathbf{m})|}q(\mathbf{m}^*,\mathbf{m})=k\right\}\\
  &=\binom{k+n-1}{n}\left(\frac{1}{2d}\right)^n\left(\frac{2d-1}{2d}\right)^k.
  \end{aligned}
  $$
  \smallskip
  \item if $\mathbf{m}\in\mathcal{M}^-(\mathbf{n})$, $d_0(\mathbf{m})>0$ and $\|\mathbf{n}\|\geq2$,  then
  $$
  \begin{aligned}
  &\mathsf{P}\left\{q(\mathbf{m},\mathbf{n})=n~\Big|~\sum_{\mathbf{m}^*\in\mathcal{M}^-(\mathbf{m})}q(\mathbf{m}^*,\mathbf{m})=k\right\}\\
  &=\binom{k+n-1}{n}\left(\frac{r(\mathbf{n})}{2d}\right)^n
  \left(\frac{2d-r(\mathbf{n})}{2d}\right)^k,
  \end{aligned}
  $$
where $r(\mathbf{n})$ is the rank of the vector $\mathbf{n}\in|\mathcal{M}^+(\mathbf{m})|$. (Definition \ref{d2} for the rank remains the same as in the two-dimensional case, and $r(\mathbf{n})$ can be either one or two. The number of vectors in the set $|\mathcal{M}^+(\mathbf{m})|$ having rank two is equal to $d_0(\mathbf{m})$. In Figure 4 this property is illustrated for $d=3$.)
  \smallskip
  \item if $\mathbf{m}\in\mathcal{M}^-(\mathbf{n})$, $d_0(\mathbf{m})=0$ and $\|\mathbf{n}\|\geq2$, then
  $$
  \begin{aligned}
  &\mathsf{P}\left\{q(\mathbf{m},\mathbf{n})=n~\Big|~\sum_{\mathbf{m}^*\in\mathcal{M}^-(\mathbf{m})}q(\mathbf{m}^*,\mathbf{m})=k\right\}\\
  &=\binom{k+n-1}{n}\left(\frac{1}{2d}\right)^n\left(\frac{2d-1}{2d}\right)^k.
  \end{aligned}
  $$
\end{enumerate}

\begin{figure}
\includegraphics[width=10cm, height=12cm]{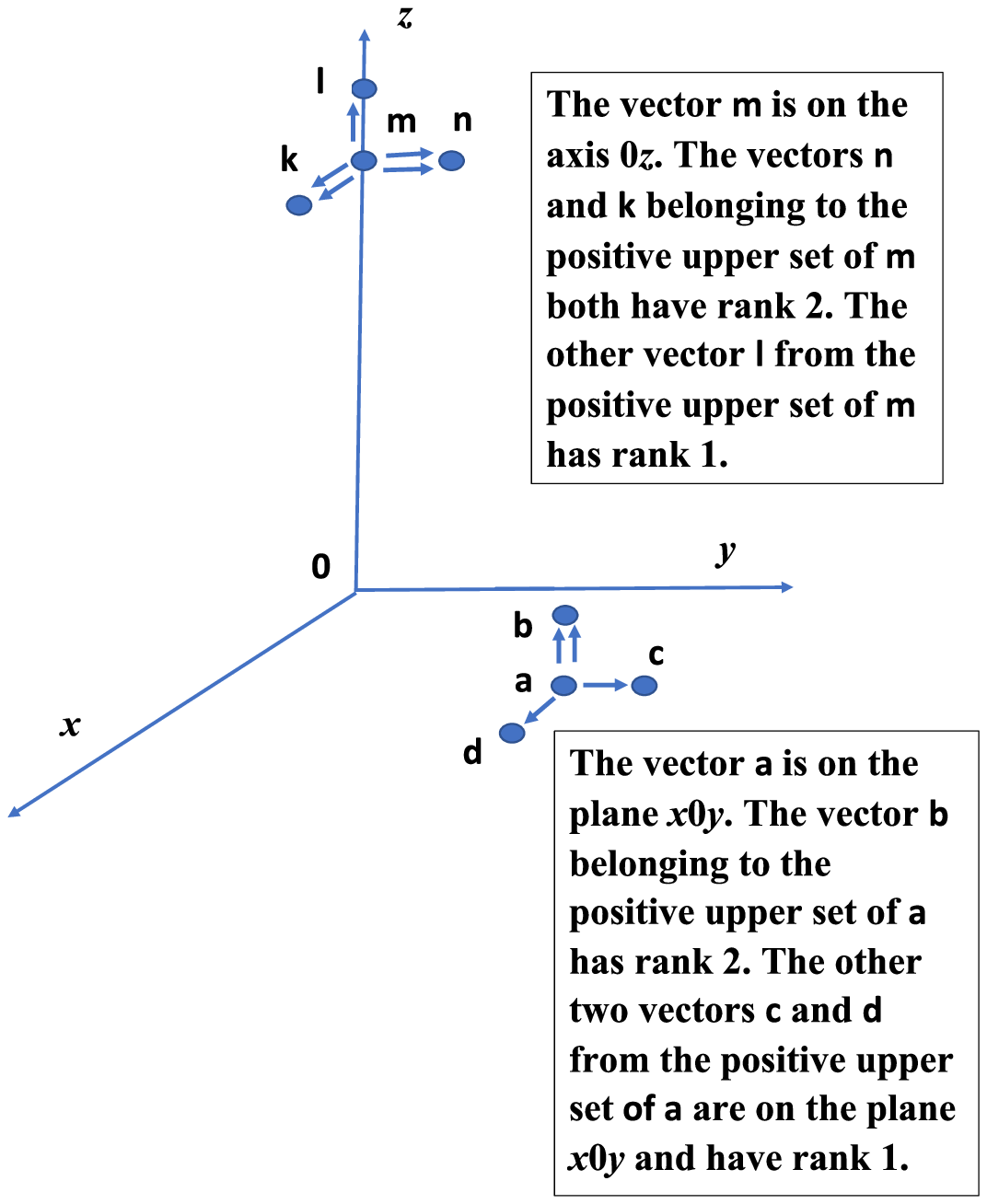}
\caption{Illustration of rank 1 and rank 2 vectors in three-dimensional random walks.}
\end{figure}

For Markov chain $R_n$ we have the properties:
\begin{enumerate}
  \item $R_0=1$.
  \item $\mathsf{P}\{R_{n}=k|R_{n-1}=m\}=\binom{k+m-1}{k}\left(\frac{C(n,d)}
{C_0(n,d)+2C(n,d)}\right)^m\left(\frac{C_0(n,d)+C(n,d)}{C_0(n,d)+2C(n,d)}\right)^k$,
where
$$
C_0(n,d)=\sum_{i=1}^{d-1}(d-i)2^{i+1}\binom{d}{i}\binom{n-1}{i-1},
$$
$$
C(n,d)=\sum_{i=1}^{d}i2^i\binom{d}{i}\binom{n-1}{i-1}.
$$
\end{enumerate}

Note, that for the random walk in $[-N,N]^d$ the upper set $\mathcal{M}^+(\mathbf{n})$ is not always defined. For instance, if $\mathbf{n}=(N,N,\ldots,N)=N\mathbf{1}$ or $\mathbf{n}=(-N,-N,\ldots,-N)=-N\mathbf{1}$ (the margin states), then $\mathcal{M}^+(\mathbf{n})=\emptyset$. In other cases, where $\mathbf{n}$ is a boundary element, the upper set $\mathcal{M}^+(\mathbf{n})$ is defined, but the number of vectors is less than $d$. For instance, if $\mathbf{n}=N\mathbf{1}_1$, then the number of vectors in $\mathcal{M}^+(\mathbf{n})$ is $d-1$. This makes the study of crossings sets and states complicated in general. However, considering the states of the random walk in smaller area such as the rectangle
$\mathcal{S}\subseteq[-(N-1), (N-1)]^d$ $(N>1)$ enables us to avoid the complication and make the extension of the earlier results smooth. We have the following result.

\begin{thm}\label{t5} For any $\mathbf{n}\in\mathcal{S}\setminus\{\mathbf{0}\}$ the statements of Theorems \ref{t2}, \ref{t1}, and \ref{t3}, in which the random objects $P_{\mathbf{n}}$, $Q_{\mathbf{n}}$ and $R_n$ are redefined in this section, hold true. Relations \eqref{eq-3.7}, \eqref{eq-3.8} and \eqref{eq-3.9} in the given case are, respectively, rewritten as follows.
\begin{equation*}
\mathsf{E}\left\{\overrightarrow{f}(\mathbf{n})\right\}=2^
{d_0(\mathbf{n})-d}
\mathsf{E}\left\{\overrightarrow{f}[\mathcal{X}(\mathbf{n})]\right\}
=\frac{d-d_0(\mathbf{n})}{2d},
\end{equation*}
\begin{equation*}
\mathsf{E}\left\{\overleftarrow{f}(\mathbf{n})\right\}=2^
{d_0(\mathbf{n})-d}
\mathsf{E}\left\{\overleftarrow{f}[\mathcal{X}(\mathbf{n})]\right\}
=\frac{d+d_0(\mathbf{n})}{2d},
\end{equation*}
and
\begin{equation*}
\mathsf{E}\left\{{f}(\mathbf{n})\right\}=2^
{d_0(\mathbf{n})-d}
\mathsf{E}\left\{{f}[\mathcal{X}(\mathbf{n})]\right\} =1.
\end{equation*}
\end{thm}

\begin{proof}
The proof of statements related to Theorems \ref{t2} and \ref{t1} is based on straightforward extension of the proof of Theorems \ref{t2} and \ref{t1} given in Sections \ref{S5} and \ref{S6}, respectively. The proof of the statement related to Theorem \ref{t3} is also an extension of the proof of corresponding Theorem \ref{t3}, but the cases where the required state $\mathbf{n}$ is on the boundary of the main subspace $[0,N]^d$ (in the case of dimension three it is the positive three dimensional subspace $xyz$) should be studied in more details.
Specifically,  in the proof of Theorem \ref{t3} the boundary case is associated with $d_0(\mathbf{n})=1$. In the present proof of the corresponding statement associated with Theorem \ref{t3}, each boundary case is associated with the inequality $1\leq d_0(\mathbf{n})\leq d-1$, and the number of the series of the intervals like those given in \eqref{eq-7.3} must be $d_0(\mathbf{n})$. Then, instead of \eqref{eq-7.4} we have:
$$
z^*_{\mathbf{n}}=z_{\mathbf{n}}+\sum_{i=1}^{d_0(\mathbf{n})}z^\prime_{\mathbf{m}^\prime_i,\mathbf{n}},
$$
where $z^\prime_{\mathbf{m}^\prime_1,\mathbf{n}}$, $z^\prime_{\mathbf{m}^\prime_2,\mathbf{n}}$,\ldots, $z^\prime_{\mathbf{m}^\prime_{d_0(\mathbf{n}),\mathbf{n}}}$ are independent and identically distributed, and $z_{\mathbf{n}}$ is given by \eqref{eq-7.6}. If $\mathbf{n}$ is from the interior, then \eqref{eq-7.5} here is true as well. In the proof of the statements related to Theorem \ref{t4}, the proof in the case when $\mathbf{n}\geq\mathbf{1}$ is the same in Theorem \ref{t4}. So, Relation \eqref{eq-8.1} holds in this case as well. Instead of \eqref{eq-8.2} we have
$$
\mathsf{E}\{f[\mathcal{X}(\mathbf{1}_i)]\}=\frac{1}{d}, \quad i=1,2,\ldots,d.
$$
Then, for $\mathbf{n}\geq\mathbf{1}$ instead of \eqref{eq-8.3} we have
\begin{equation}\label{eq-9.1}
\mathsf{E}\{f(\mathbf{n})\}=\frac{1}{2d}\mathsf{E}\{f[\mathcal{X}(\mathbf{1})]\},
\end{equation}
and similarly to that in the proof of Theorem \ref{t4} (see \eqref{eq-8.4})
\begin{equation}\label{eq-9.2}
\mathsf{E}\{f(\mathbf{1}_i)\}=\mathsf{E}\{f(-\mathbf{1}_i)\}=\frac{1}{2d}\mathsf{E}\{f[\mathcal{X}(\mathbf{1})]\}, \quad i=1,2,\ldots,d.
\end{equation}
Using Relation (A.4), in fact we obtain more general result than \eqref{eq-9.1} and \eqref{eq-9.2}. Namely, \eqref{eq-9.1} is true for all $\mathbf{n}\in\mathcal{S}\setminus\{\mathbf{0}\}$. That is, similarly to \eqref{eq-8.11} we have
\begin{equation}\label{eq-9.3}
\mathsf{E}\{f(\mathbf{n})\}=c
\end{equation}
for all $\mathbf{n}\in\mathcal{S}\setminus\{\mathbf{0}\}$. The next part of the proof is similar to that of Theorem \ref{t4}. Since all states $\mathbf{n}$ are equally likely, from the total expectation formula we have:
\begin{equation}\label{eq-9.4}
\mathsf{E}\left\{\overrightarrow{f}(\mathbf{n})\right\}=\frac{\text{The number of vectors in}\ \mathcal{M}^-(\mathbf{n})}{2d}=\frac{d-d_0(\mathbf{n})}{2d},
\end{equation}
\begin{equation}\label{eq-9.5}
\mathsf{E}\left\{\overleftarrow{f}(\mathbf{n})\right\}=\frac{\text{The number of vectors in}\ \mathcal{M}^+(\mathbf{n})}{2d}=\frac{d+d_0(\mathbf{n})}{2d}.
\end{equation}
From \eqref{eq-9.4} and \eqref{eq-9.5} we finally arrive at the conclusion that the constant $c$ in \eqref{eq-9.3} is equal to one.
\end{proof}

\section{On random walks in $\mathbb{Z}^d$, $d>2$.}\label{S10}

In this section we discuss crossing states in random walks in $\mathbb{Z}^d$, $d>2$. The random walks of dimension three or higher are transient, therefore the question about crossing states makes sense under the condition that it is eventually returns to the origin. Let $\mathfrak{A}$ denote the event ``the
random walk that starts from the origin returns to the original point again". Under this condition, a natural question is as follows. Let $\mathbf{n}\in\mathbb{Z}^d\setminus\{\mathbf{0}\}$. Is $\mathsf{E}\{f(\mathbf{n})~|~\mathfrak{A}\}=1$ true?

Unfortunately, the condition $\mathfrak{A}$ makes the approach suggested in the present paper impossible to address this question. The level-crossing approach that was first considered in Section \ref{S3} for one-dimensional random walk and then developed for two-dimensional random walk and multidimensional random walk of an arbitrary dimension $d$ but defined in bounded areas assumes that the consecutive intervals such as \eqref{eq-6.7} have a regenerative structure, and the number of merged points in these intervals do not depend on $\|\mathbf{n}\|$. Under the condition $\mathfrak{A}$ this dependence exists. If the state $\mathbf{n}$ is ``closer to zero", then the probability to return to the origin is higher compared to the case when the state $\mathbf{n}$ is ``far from zero". The set of sample paths that starts from state $\mathbf{n}$ and returns to the same state $\mathbf{n}$ in the case when $\mathbf{n}$ is ``closer to zero" is richer compared to the set of sample paths of the similar time period when $\mathbf{n}$ is ``far from zero". Furthermore, for $d\geq5$ the $d$-dimensional random walk is strongly transient \cite{Hughes}, that is, the conditional expected length of a sojourn time from the origin to the origin
given $\mathfrak{A}$ is finite. This means that $\mathsf{E}\{f(\mathbf{n})~|~\mathfrak{A}\}=1$ cannot be true for $d\geq5$ and probably is not true for $d=3$ or $d=4$ either.

Our conjecture in the relation to this case is as follows. For random walks defined in $\mathbb{Z}^d$, $d>2$, the only inequality $\mathsf{E}\{f(\mathbf{n})~|~\mathfrak{A}\}<1$ is true ($\mathbf{n}\in\mathbb{Z}^d\setminus\{\mathbf{0}\}$). Furthermore, $\mathsf{E}\{f(\mathbf{n})~|~\mathfrak{A}\}$ decreases as $\|\mathbf{n}\|$ increases, and if $d\geq5$, then $\mathsf{E}\{f(\mathbf{n})~|~\mathfrak{A}\}$ vanishes as $\|\mathbf{n}\|$ tends to infinity.

\section{Numerical study}\label{S_num} In this section we provide numerical results for level-crossings of three different random walks considered in the paper: one-dimensional random walks, two-dimensional random walk and three-dimensional random walk in certainly defined bounded area. In our simulations the only random walks that eventually returned to the origin after no more than $100,000$ steps were counted. Each output result is obtained based on $1,000,000$ runs.

\subsection{One-dimensional random walk} It turns out that for this random walk only $997,482$ runs out of $1,000,000$ are  successfully terminated. That is, $997482$ simulated random walks eventually returned to the original point for no more than $100,000$ steps. Simulation shows stable results in general. We studied numerically the number of crossings levels $1$, $2$, $3$, $5$, $10$, $20$, $50$ and $100$ that are reflected in Table \ref{T1}.
\begin{table}
    \begin{center}
        \caption{Level-crossings in the one-dimensional random walk}\label{T1}
            \begin{tabular}{||c|c|c|c||}
            \hline
             Level $n$ & Estimated $\mathsf{E}f(n)$ & Estimated $\mathsf{E}\overrightarrow{f}(n)$ & Estimated $\mathsf{E}\overleftarrow{f}(n)$\\
            \hline\hline
 1 & 0.9962 & 0.5000 & 0.4964\\
 2 & 0.9898 & 0.4964 & 0.4935\\
 3 & 0.9848 & 0.4935 & 0.4913\\
 5 & 0.9759 & 0.4896 & 0.4863\\
 10& 0.9568 & 0.4795 & 0.4773\\
 20& 0.9111 & 0.4563 & 0.4547\\
 50& 0.7542 & 0.3780 & 0.3763\\
100& 0.5336 & 0.2679 & 0.2657\\
            \hline
            \end{tabular}
    \end{center}
\end{table}

It is seen from Table \ref{T1} that the results for the crossings levels from $1$ to $20$ show the results that are close to theoretical. The results for levels $50$ and $100$ are essentially deviated. Notice that with increasing $n$, the obtained numerical values for $\mathsf{E}f(n)$, $\mathsf{E}\overrightarrow{f}(n)$ and $\mathsf{E}\overleftarrow{f}(n)$ decrease in $n$. The reason of this decreasing and the biased values for $n=50$ and $n=100$ can be explained by the fact that only $99.7\%$ of simulation runs are successfully terminated. A more detailed explanation is given for the case of the two-dimensional random walk in the section below.

\subsection{Two-dimensional random walk} Two-dimensional random walk is central to this paper. However, the numerical analysis is not so successful as in one-dimensional case. If in the one-dimensional case the number of successfully terminated random walks was about $99.7\%$, in the case of the two-dimensional random walk the number of successfully terminated random walks is approximately $78\%$. This percent cannot be substantially improved by increasing the number of steps. Simulation will become time consuming with no visible result. In addition, the variance of the random fields increases more rapidly as $\|\mathbf{n}\|$ increases to infinity compared to the branching process with increasing $n$ to infinity. For these reasons, the numerical results in this case are considered for the only states located close to the origin.

\begin{table}
    \begin{center}
        \caption{State-crossings in the two-dimensional random walk}\label{T3}
            \begin{tabular}{||c|c|c|c||}
            \hline
             State $\mathbf{n}$ & Estimated $\mathsf{E}f(\mathbf{n})$ & Estimated $\mathsf{E}\overrightarrow{f}(\mathbf{n})$ & Estimated $\mathsf{E}\overleftarrow{f}(\mathbf{n})$\\
            \hline\hline
 ($0$ $-1$) & 0.7823 & 0.2507 & 0.5316\\
 ($0$  $1$) & 0.7779 & 0.2488 & 0.5291\\
 ($1$  $0$) & 0.7798 & 0.2511 & 0.5287\\
 ($-1$ $0$) & 0.7780 & 0.2494 & 0.5286\\
 ($1$ $1$) & 0.6625 & 0.3584 & 0.3042\\
($-1$ $-1$)& 0.6637 & 0.3605 & 0.3032\\
($1$ $-1$) & 0.6642 & 0.3600 & 0.3042\\
($-1$ $1$) & 0.6633 & 0.3598 & 0.3035\\
            \hline
            \end{tabular}
    \end{center}
\end{table}

Now we discuss the obtained results in Table \ref{T3} starting first from the states ($0$ $-1$), ($0$ $1$), ($1$ $0$) and ($-1$ $0$). The obtained frequencies of up-directed crossings  in all of the four cases are close to the theoretical value of the mean $0.25$. However, the frequencies of down-directed crossings are essentially biased. The last can be explained as follows. Longer random walks having length of more than $100,000$ steps affect on the number of down-directed crossings essentially. Absence of a large number of long random walks from the simulation series thus essentially decreases the number of down-directed crossings compared the number that should be. The cut down number of down-directed crossings in turn changes the total number of undirected crossings as well. A similar picture is regarding the states ($1$ $1$), ($-1$ $-1$), ($1$ $-1$) and ($-1$ $1$). The fraction of up-directed crossings is slightly greater than the fraction of down-directed crossings because of the missing longer simulation series. As well, the missing longer series affecting the number of down-directed crossings of the states ($0$ $-1$), ($0$ $1$), ($1$ $0$) and ($-1$ $0$) in turn affect to the number of up-directed crossings the states ($1$ $1$), ($-1$ $-1$), ($1$ $-1$) and ($-1$ $1$) since these states are neighbors with respect to the aforementioned states.

\subsection{Three-dimensional random walk in a bounded region} In this numerical example we simulate the numbers of crossings states in $[-5, 5]^3$. In this case, all the random walks were successfully terminated. Table \ref{T4} presents the simulation results. The obtained results are close to the obtained theoretical results.

\begin{table}
    \begin{center}
        \caption{State-crossings in a three-dimensional random walk in a bounded region}\label{T4}
            \begin{tabular}{||c|c|c|c||}
            \hline
             State $\mathbf{n}$ & Estimated $\mathsf{E}f(\mathbf{n})$ & Estimated $\mathsf{E}\overrightarrow{f}(\mathbf{n})$ & Estimated $\mathsf{E}\overleftarrow{f}(\mathbf{n})$\\
            \hline\hline
 ($0$ $0$ $2$) & 1.0017 & 0.1665 & 0.8352\\
 ($0$  $2$ $2$) & 1.0039 & 0.3348 & 0.6691\\
 ($0$  $0$ $-2$) & 1.0034 & 0.1671 & 0.8363\\
 ($0$ $-2$ $-2$) & 1.0018 & 0.3340 & 0.6678\\
 ($2$ $2$ $2$) & 0.9989 & 0.4991 & 0.4998\\
($-2$ $-2$ $-2$)& 0.9983 & 0.4986 & 0.4997\\
            \hline
            \end{tabular}
    \end{center}
\end{table}

\section{Summary of the results and discussion}\label{S11}

In this paper, a comprehensive analysis of the numbers of
state-crossings and the numbers of crossings specifically defined
sets of states is provided. The explicit representations for
probability distributions and expectations of the numbers of
crossings the states
$\mathbf{n}\in\mathbb{Z}^d\setminus\{\mathbf{0}\}$ and the sets of
states $\mathcal{X}(\mathbf{n})$ and $\mathcal{N}(\|\mathbf{n}\|)$
are obtained.

\begin{table}
    \begin{center}
            \caption{Comparison table for the property of conditional
            expectations of the numbers of state-crossings and crossings the sets
            of states}\label{T2}
        \begin{tabular}{||c|c||}\hline
        One-dimensional & Two-dimensional \\
         random walk    &  random walk     \\
        \hline
         The expectations & The conditional expectations  \\
         are the same for & are the same for all    \\
         all values $n$, $n\neq0$   &  vectors $\mathbf{n}$, $\mathbf{n}\neq\mathbf{0}$\\
         \textit{(true, false)} &\textit{(true, false)}\\
        \hline\hline
        & \\
for $\overrightarrow{f}(n)$ -- \textit{true}& for
$\overrightarrow{f}(\mathbf{n})$ -- \textit{false} in general\\
& and \textit{true} for $|\mathbf{n}|\geq\mathbf{1}$;\\
&
for $\overrightarrow{f}[\mathcal{N}(n)]$ -- \textit{false}\\
\hline &\\
for $\overleftarrow{f}(n)$ -- \textit{true}& for
$\overleftarrow{f}(\mathbf{n})$ -- \textit{false} in general\\
& and \textit{true} for $|\mathbf{n}|\geq\mathbf{1}$;\\
& for $\overleftarrow{f}[\mathcal{N}(n)]$ -- \textit{false}\\
\hline &\\
         for $f(n)$ -- \textit{true} & for $f(\mathbf{n})$ --
         \textit{true};\\
         &
for ${f}[\mathcal{N}(n)]$ -- \textit{false}\\
\hline &\\
         for $\overrightarrow{f}(\{-n,n\})$ -- \textit{true}
         &for $\overrightarrow{f}[\mathcal{X}(\mathbf{n})]$ -- \textit{false} in general\\
         & and \textit{true} for all $\mathbf{n}\geq\mathbf{1}$\\
\hline &\\
         for $\overleftarrow{f}(\{-n,n\})$ -- \textit{true}
         &for $\overleftarrow{f}[\mathcal{X}(\mathbf{n})]$ -- \textit{false} in general\\
         & and \textit{true} for all $\mathbf{n}\geq\mathbf{1}$\\
\hline &\\
         for ${f}(\{-n,n\})$ -- \textit{true}
         &for ${f}[\mathcal{X}(\mathbf{n})]$ -- \textit{false} in general\\
         & and \textit{true} for all $\mathbf{n}\geq\mathbf{1}$\\
\hline
        \end{tabular}
    \end{center}
\end{table}

In Table \ref{T2}, we survey the properties of the conditional expectations
for the numbers of state-crossings and crossings sets of states. We
give their comparison table for one-dimensional and two-dimensional
random walks. Specifically, we indicate whether or not the
aforementioned conditional expectations are the same. For instance,
the record ``for $\overrightarrow{f}(n)$ -- \textit{true}" means
that $\mathsf{E}\big\{\overrightarrow{f}(n)\big\}=c$ for all $n$ integer,
$-\infty<n<\infty$, where $c$ is the same constant for all $n$. (In some cases it is implied that $c=1$. However, for $\mathsf{E}\big\{\overrightarrow{f}(n)\big\}$ or $\mathsf{E}\big\{\overleftarrow{f}(n)\big\}$ related to the one-dimensional random walk the constant $c$ is a half.
It is also half for $\mathsf{E}\big\{\overrightarrow{f}(\mathbf{n})\big\}$ or $\mathsf{E}\big\{\overleftarrow{f}(\mathbf{n})\}$ for $|\mathbf{n}|\geq\mathbf{1}$.
Illustrating the only qualitative properties, we ignore this difference.)
In another record ``for
$\overrightarrow{f}(\mathbf{n})$ -- \textit{false} in general and
\textit{true} for $|\mathbf{n}|\geq\mathbf{1}$" means that
$\mathsf{E}\big\{\overrightarrow{f}(\mathbf{n})\big\}=c$ is incorrect in general
when $\mathbf{n}\in\mathbb{Z}^2\setminus\{\mathbf{0}\}$, but correct
if the subset $|\mathbf{n}|\geq\mathbf{1}$ is considered. Note, that
$\overrightarrow{f}(\{-n,n\})$, $\overleftarrow{f}(\{-n,n\})$ and
$f(n)$ are the one-dimensional analogues of
$\overrightarrow{f}[\mathcal{X}(\mathbf{n})]$,
$\overleftarrow{f}[\mathcal{X}(\mathbf{n})]$ and
$f[\mathcal{X}(\mathbf{n})]$, respectively,  or
$\overrightarrow{f}[\mathcal{N}(n)]$,
$\overleftarrow{f}[\mathcal{N}(n)]$ and $f[\mathcal{N}(n)]$,
respectively. The table does not include the results of Section \ref{S9} for multidimensional random walks in bounded areas considering them as specific.

The basic property of the paper that \textit{the expected number of state-crossings
$\mathbf{n}\in\mathbb{Z}^2\setminus\{\mathbf{0}\}$ is the same} can
be extended to a wider class of random walks considered in
\cite{Abramov_AoP}. For instance, for the family of conservative
random walks in \cite{Abramov_AoP} indicated as \textit{symmetric
random walks} (Model 1), the property is satisfied since the results in the Appendix hold true in the case of symmetric random walks as well.

Recall that the
aforementioned family of symmetric random walks is described by the system of
equations
\begin{eqnarray*}
\mathbf{S}_0&=&\mathbf{0},\label{eq-10.1}\\
\mathbf{S}_{t}&=&\mathbf{S}_{t-1}+\widetilde{\mathbf{e}}_t, \quad
t\geq1,\label{eq-10.2}
\end{eqnarray*}
where the vector $\widetilde{\mathbf{e}}_t$ is defined as follows.
As in \eqref{eq-1.1} and \eqref{eq-1.2}, it is one of the
vectors $\{\pm\mathbf{1}_{i}, i=1,2\}$ that is randomly
chosen. But the probability of this choice depend on $i$. That is,
for the vector $\mathbf{1}_i$ (or the vector $(-\mathbf{1}_i)$) to
be chosen this probability is $\alpha_i>0$, $i=1,2$, and
$2(\alpha_1+\alpha_2)=1$.

\appendix

\section{Appendix: Derivation of the formula for the transition probability from $\mathcal{N}(n)$ to $\mathcal{N}(n+1)$ for the random walk in $\mathbb{Z}^d$}

The results of the appendix are taken from \cite{Abramov_AoP}, pp. 1906--1908.

\subsection*{Description of the model} The random walk in $\mathbb{Z}^d$ is modelled by the $d$ independent queueing processes
as follows. Assume that arrivals in the $i$th queueing system are Poisson with rate $\lambda$,
and service times are exponentially distributed with the parameter $\lambda$. If a system
becomes free, it is switched for a special service, which is exponentially distributed with the same parameter $\lambda$. This service is
negative, and it results in a new customer in the queue. If during a negative service a
new arrival occurs, the negative service remains unfinished and not resumed.
The negative service models the reflection at zero and in fact implies the state-dependent
arrival rate, which becomes equal to $2\lambda$ at the moment when the system is
empty. It is associated with the situation, when an original one-dimensional random
walk reaches zero at some time moment $s$, and at the next time moment $s + 1$ it must
take one of the values $\pm1$ that corresponds to value $+1$ for an one-dimensional random
walk reflected at zero.

\subsection*{Queueing systems with finite capacity and characterization of the level $n$ probability} Assume that the number of waiting places in each of the queueing system is $N$. The assumption on limited number of waiting places means that an
arriving customer, who meets N customers in the system, is lost. For a vector $\mathbf{n}=\big(n^{(1)}, n^{(2)},\ldots, n^{(d)}\big)$ satisfying $\|\mathbf{n}\|<N$ let $P_N(\mathbf{n})$ denote the stationary probability to be in state $\mathbf{n}$ immediately before arrival of a customer in one of the $d$ queueing systems. Application of the PASTA property \cite{Wolff1} enables us to first obtain the stationary probability for each single system to obtain then the required stationary probability $P_N(\mathbf{n})$. Let $P_N^{(i)}(n)$ denote the stationary probability to be in state $n<N$ for the $i$th queueing system. We have
$$
P_N^{(i)}(n)=\begin{cases}\frac{2}{2N+1}, &\text{for} \quad 1\leq n\leq N,\\
\frac{1}{2N+1}, &\text{for} \quad n=0.
\end{cases}
$$
Hence the queueing systems are independent, the product form solution for $P_N(\mathbf{n})$ is
\begin{equation}\label{A.1}
P_N(\mathbf{n})=2^{d-d_0(\mathbf{n})}\frac{1}{(2N+1)^d},
\end{equation}
where $d_0(\mathbf{n})$ denotes the number of zero components in the presentation of the vector $\mathbf{n}$.

The total number
of vectors having norm $n$ in $\mathbb{Z}^d_+$
is
\begin{equation}\label{A.2}
\sum_{i=1}^{d}\binom{d}{i}\binom{n-1}{i-1}.
\end{equation}
(The formula sums over $i$ being the number of nonzero components
of the vector.)
Hence, denoting the stationary state probability to belong to the set $\mathcal{N}^+(n)$ by
$P_N[\mathcal{N}^+(n)]$, we have
\begin{equation}\label{A.3}
P_N[\mathcal{N}^+(n)]=\sum_{\mathbf{n}\in\mathcal{N}^+(n)}=\frac{1}{(2N+1)^d}\sum_{i=1}^{d}2^i\binom{d}{i}\binom{n-1}{i-1},
\end{equation}
where the term
$$
\sum_{i=1}^{d}2^i\binom{d}{i}\binom{n-1}{i-1}
$$
on the right-hand side of \eqref{A.3} characterizes the total number of vectors in $\mathbb{Z}^d$ having norm $n$.

It follows from \eqref{A.1} that for any two vectors $\mathbf{n}_1$ and $\mathbf{n}_2$ satisfying $\|\mathbf{n}_1\|<N$ and $\|\mathbf{n}_2\|<N$ we have
\begin{equation}\label{A.4}
\frac{P_N(\mathbf{n}_1)}{P_N(\mathbf{n}_2)}=2^{d_0(\mathbf{n}_2)-d_0(\mathbf{n}_1)}
\end{equation}
independently on $N$. Hence,
\begin{equation}\label{A.5}
\lim_{N\to\infty}\frac{P_N(\mathbf{n}_1)}{P_N(\mathbf{n}_2)}=2^{d_0(\mathbf{n}_2)-d_0(\mathbf{n}_1)}.
\end{equation}
For  $N=\infty$, let $P_\infty(\mathbf{n},\tau)$ be the probability that at time $\tau$ the system with infinite capacity is in state $\mathbf{n}$.
Then, for the limiting ratio of the final probabilities we also have
\begin{equation}\label{A.6}
\lim_{\tau\to\infty}\frac{P_\infty(\mathbf{n}_1,\tau)}{P_\infty(\mathbf{n}_2,\tau)}=2^{d_0(\mathbf{n}_2)-d_0(\mathbf{n}_1)}.
\end{equation}
The result in \eqref{A.6} is true,
since the limit in \eqref{A.5} due to Relation \eqref{A.4} is uniform, and interchanging the order of limits $\tau$ vs $N$ is correct. (The last can also be established directly \cite{Abramov_AoP}.)

Let $p_n(d)$ denote the transition probability from the set of states $\mathcal{N}^+
(n)$ (level $n$)
to the set of states $\mathcal{N}^+
(n+1)$ (level $n + 1$), and let $q_n(d)$ = $1-p_n(d)$ denote the
transition probability from the level $n$ to the level $n-1$.
Derive the formula for $p_n(d)$.

The total number of vectors in the set $\mathcal{N}^+
(n)$ is given by \eqref{A.2}. Each vector contains
$d$ components. Hence, the total number of components in the set of vectors in $\mathcal{N}^+
(n)$
is
\begin{equation}\label{A.7}
d\sum_{i=1}^{d}\binom{d}{i}\binom{n-1}{i-1}.
\end{equation}
Among them, the total number of zero components is
\begin{equation}\label{A.8}
\sum_{i=1}^{d-1}(d-i)\binom{d}{i}\binom{n-1}{i-1},
\end{equation}
and the total number of nonzero components is
\begin{equation}\label{A.9}
\sum_{i=1}^{d}i\binom{d}{i}\binom{n-1}{i-1}.
\end{equation}
To derive the formula for $p_n(d)$ let us build the sample space. The components of all
vectors in $\mathcal{N}^+
(n)$, the total number of which is given by \eqref{A.7}, are not equally likely.
According to \eqref{A.1}, a nonzero component appears with two times higher probability
than a zero component. To make the components equally likely, we are to extend the
number of nonzero components by factor 2. Then the total number of equally likely
components is to be equal to
\begin{equation}\label{A.10}
d\sum_{i=1}^{d}2^i\binom{d}{i}\binom{n-1}{i-1}.
\end{equation}

Following \eqref{A.10}, the sample space for level $n$ contains
$$
2d\sum_{i=1}^{d}2^i\binom{d}{i}\binom{n-1}{i-1},
$$
that is two times more than that given by \eqref{A.10}. This is because it include possible transitions from each state (component) in the two directions. Specifically, let
$C_0(n, d)$ denote the number of possible transitions associated with reflections at zero (the number of zero components given in \eqref{A.8} multiplied by two), and let $2C(n, d)$ be the rest of transitions, the half of which are associated with the transitions from level $n$ to $n+1$ and another half with the transitions from level $n$ to $n-1$. So, $C(n, d)$ is the number of nonzero components given in \eqref{A.9}

That is, the expression in \eqref{A.9} is presented as
\begin{equation}\label{A.11}
2d\sum_{i=1}^{d}2^i\binom{d}{i}\binom{n-1}{i-1}=C_0(n, d)+2C(n, d).
\end{equation}
Then, the total number of transitions from level $n$ to $n+1$ is $C_0(n, d)+C(n, d)$ and the total number of transitions from level $n$ to $n-1$ is $C(n,d)$. Then, from \eqref{A.8}, \eqref{A.9} and \eqref{A.11} we obtain
$$
\begin{aligned}
p_n(d)&=\frac{C_0(n,d)+C(n,d)}{C_0(n,d)+2C(n,d)}\\
&=\frac{2\sum_{i=1}^{d-1}(d-i)2^i\binom{d}{i}\binom{n-1}{i-1}+\sum_{i=1}^{d}i2^i\binom{d}{i}\binom{n-1}{i-1}}
{2\sum_{i=1}^{d-1}(d-i)2^i\binom{d}{i}\binom{n-1}{i-1}+2\sum_{i=1}^{d}i2^i\binom{d}{i}\binom{n-1}{i-1}}.
\end{aligned}
$$


\end{document}